\documentclass[11pt]{article}
\usepackage{amsfonts}
\usepackage{amsmath,amssymb}
\usepackage{multicol}
\usepackage{color}
\newtheorem{theorem}{Theorem}[section]
\newtheorem{lemma}{Lemma}[section]

\newtheorem{definition}{Definition}[section]
\newtheorem{remark}{Remark}[section]

\numberwithin{equation}{section}
\newenvironment{proof}[1][Proof]{\noindent\textbf{#1.} }{\hfill $\Box$}

\begin{document}
\title{\LARGE\bf{ Infinitely many solutions for a  Schr$\ddot{\text{o}}$dinger equation with  sign-changing potential and nonlinear term}
%\footnote{This work was supported by NFSC under Grants No. 11931012, 11871387 and partially by the Zhejiang Provincial Natural Science Foundation %under Grant No. LQ21A010014.
% Emails:
%gulongjiang0@163.com,  hszhou@whut.edu.cn.  }
  }
\date{}
 \author{ Long-Jiang Gu$\thanks{Supported by NFSC under Grant No. 12101577 and  the Zhejiang Provincial Natural Science Foundation under Grant No. LQ21A010014 ,
Email: gulongjiang0@163.com}$    ,  \   \  Huan-Song Zhou$\thanks{Supported by NFSC under Grants No. 11931012, 11871387,
Email: hszhou@whut.edu.cn (Corresponding author)}$  \\
{\small $^{1}$ School of Mathematics and Physics, China University of Geosciences}\\
         {\small Wuhan 430074, China}\\
           {\small $^{2}$ Center for Mathematical Sciences and Department of Mathematics, Wuhan University of Technology}\\
          {\small Wuhan 430070, China}\\
                                } \maketitle
\begin{center}
\begin{minipage}{13cm}
\par
 \small  {\bf Abstract:}  We propose a new variational approach to finding multiple critical points for strongly indefinite problems without assuming the weak upper semicontinuity on the variational functionals.
By this approach, we obtain the existence of  infinitely many geometrically distinct solutions for a stationary periodic  Schr\text{$\ddot{\text{o}}$}dinger  equation, in which
the linear part is strongly indefinite and the nonlinear term is allowed to change sign in  general ways.


\end{minipage}
\end{center}

 {\section{Introduction}}
 \setcounter{equation}{0}

 \par
 In this paper,  we are
%% going to find a new way for multiple solutions of strongly indefinite problems
%% without the weak upper semicontinuity assumption. To illustrate the method, we focus on
concerned with the following well-studied stationary periodic Schr\text{$\ddot{\text{o}}$}dinger equation
\begin{equation}\label{5.equation}
      \left\{
   \begin{array}{ll}
     -\Delta u + V(x)u=f(x,u), \quad x=(x_1,x_2,...,x_N)\in \mathbb{R}^{N}, \\
     u\in H^{1}(\mathbb{R}^{N}) , N\geq2 ,
   \end{array}
   \right.
\end{equation}
where  the potential $V$ and the nonlinear term $f$ are sign-changing functions  satisfying
\begin{description}
  \item[$(\text{V})$]  $V \in C(\mathbb{R}^{N},\mathbb{R})$ is
 1-periodic in $x_{1},...,x_{N}$ and  0 lies in a gap of the spectrum of $-\Delta+V$, i.e.,
   $ \sup\{\sigma(-\Delta+V)\cap (-\infty,0)\}<0<\inf\{\sigma(-\Delta+V)\cap (0,+\infty)\}$.

  \item[$(\text{f}_{\text{1}})$] $f\in C(\mathbb{R}^N\times \mathbb{R})$ is 1-periodic in $x_{1},...,x_{N}$, and there exist constants $C>0$ and
         $p \in (2,2^*) $ with $2^*=
         %\left\{
   %\begin{array}{ll}
     \frac{2N}{N-2} \text{ if }  N\geq3, %\\
    \text{ and } 2^*= +\infty \text{ if }  N=2,
   %\end{array}
   %\right.
   $
   such that
   \begin{equation*}
    |f(x,t)|\leq C(1+|t|^{p-1}), \text{ for any } (x,t)\in \mathbb{R}^N\times \mathbb{R}.
   \end{equation*}
\item[$(\text{f}_{\text{2}})$]   $\lim \limits_{t\rightarrow0} {f(x,t)}/{t}=0$   uniformly in $x\in \mathbb{R}^N$.
\end{description}

The associated variational functional of equation (\ref{5.equation}) is defined by
\begin{equation}\label{5.energy1}
 \varphi(u):=\frac{1}{2}\int_{\mathbb{R}^{N}}  | \nabla u|^{2}+V(x) u ^{2} dx
    -\int_{\mathbb{R}^{N}} F(x,u)dx, \quad u \in H^{1}(\mathbb{R}^{N}),
\end{equation}
where $F(x,u):=\int_{0}^{u}f(x,s)ds$.
$\varphi (u)$ is well defined in $H^{1}(\mathbb{R}^{N})$
and is of class $\mathcal{C}^{1}$ under the conditions $(\textbf{V})$, $(\textbf{f}_{\textbf{1}})$ and $(\textbf{f}_{\textbf{2}})$.
Moreover,  for each $u\in H^{1}(\mathbb{R}^{N})$,
\begin{equation}\label{5.derivative1}
  \varphi'(u)\phi =\int_{\mathbb{R}^{N}}\nabla u \cdot \nabla\phi dx
+\int_{\mathbb{R}^{N}}V(x)u \phi dx-\int_{\mathbb{R}^{N}}f(x,u)\phi dx,\text{ for any } \phi\in H^{1}(\mathbb{R}^{N}),
\end{equation}
and $\varphi' (u)$ is weakly sequentially continuous(see \cite[Theorem A.2]{MW}).

We say $u\in H^{1}(\mathbb{R}^{N})\setminus\{0\}$ is a nontrivial weak solution of (\ref{5.equation}) if  $\varphi'(u)\phi=0$ for any $\phi\in H^{1}(\mathbb{R}^{N})$.
Two solutions of (\ref{5.equation}), $u_1$ and $u_2$,  are called {\em geometrically distinct} if $u_1$ and $u_2$ are distinct under $\mathbf{Z}^N-$translation, i.e.,  $u_1(x)\neq u_2(x+a)$  for any $a\in \mathbf{Z}^N$.

Under condition $(\textbf{V})$, the continuity and periodicity of  $V$ ensure that
the Schr\"{o}dinger operator $L:=-\Delta+V(x)$
has only continuous spectrum in $L^{2}(\mathbb{R}^{N})$  (see \cite[Theorem XIII.100]{Reed-Simon}). Since 0 lies in a gap of the spectrum of $L$,
 the Hilbert space $X:=\mathcal{D}(|L|^{\frac{1}{2}})$ can be decomposed into
 $X=X^-\oplus X^+$
such that the quadratic form:
\begin{equation*}
u\rightarrow \int_{\mathbb{R}^{N}}
 | \nabla u|^{2}+V(x) u^{2} dx,
\end{equation*}
is positive definite on $X^+$ and negative definite on  $X^-$, both $X^+$ and $X^-$
are infinite-dimensional. Such kind of functional is usually called strongly indefinite, and the associated variational problem is called {\em strongly indefinite problem}.

The inner product on $X=\mathcal{D}(|L|^{\frac{1}{2}})$ is defined by
\begin{equation}\label{defofnewnorm}
    \langle  u , v\rangle:=    \langle |L|^{\frac{1}{2}}u,   |L|^{\frac{1}{2}}v\rangle _{L^2}, \text{ for any } u,v \in X,
\end{equation}
where $\langle \cdot,\cdot \rangle_{L^2}$ is the usual inner product in $L^{2}(\mathbb{R}^{N})$.
The associated norm  for $X$ is 
\begin{equation*}\label{defofnewnorm2}
    \|u\|:=\langle |L|^{\frac{1}{2}}u,   |L|^{\frac{1}{2}}u \rangle ^{\frac{1}{2}}_{L^2}.
\end{equation*}
By a similar discussion  to  the appendix in \cite{cos}, it follows from $V(x)\in L^{\infty}(\mathbb{R}^{N})$ that  $X:=\mathcal{D}(|L|^{\frac{1}{2}})=H^{1}(\mathbb{R}^{N})$ and  the norms $\| \cdot\|$ and $\| \cdot\|_{H^1}$ are equivalent.
Moreover, $X^+$ and $X^-$ are also orthogonal with respect to $ \langle \cdot,\cdot\rangle$.

Let
\begin{equation}\label{defoforthogonalprojections}
    P :X\rightarrow X^-   \text{ and } Q :X \rightarrow X^+
\end{equation}
be the orthogonal projections,
then (\ref{5.energy1}) and (\ref{5.derivative1}) can be rewritten as
\begin{equation}\label{5.energy2}
 \varphi(u):=\frac{1}{2}(-\| Pu \|^{2}+\| Qu \|^{2})
    - \int_{\mathbb{R}^{N}} F(x,u)dx,
\end{equation}
and
\begin{equation}\label{5.derivative2}
 \varphi'(u)\phi= \langle Qu,\phi  \rangle - \langle Pu,\phi \rangle     -\int_{\mathbb{R}^{N}}f(x,u)\phi dx.
\end{equation}

In order to study the strongly indefinite variational problems as above, various variational methods have been developed
 in the last two decades(\cite{AlamaLi1,AlamaLi2,BartschDingNon-Metrizable,SK,Mederski1,Mederski2,SzulkinWeth}).
%in  papers \cite{SK,BartschDingNon-Metrizable}.
In paper \cite{SK}, Kryszewski and Szulkin  introduced the so-called  $\tau$-topology in a Hilbert space (see section 2 for details) and then established a generalized linking theorem for
\textbf{$\tau$-upper semi-continuous  functionals}.
By this linking theorem, they proved the existence of solutions for equation (\ref{5.equation}) under
 $(\textbf{V})$,   $\mathbf{(f_1)}$, $\mathbf{(f_2)}$ and the (AR) condition:
\begin{equation}
   0<\gamma F(x,t)\leq t f(x,t) \text{ for some }  \gamma>2  \text{ and for all  } x \in \mathbb{R}^N, t \in \mathbb{R}\backslash\{0\}. \tag{AR}
\end{equation}
In papers \cite{BartschDingNon-Metrizable,Bartsch-Ding}, Bartsch and Ding established some critical point theories
 in general Banach space(see also \cite{Ding}).
Thereafter, based on the $\tau$-upper semi-continuity  assumption, a series of critical point theorems were developed (see  
\cite{BC2,Ligongbao,SI1,Rup,weak-linking})
 and used to study the existence of  solutions for equation (\ref{5.equation}) with some special nonlinearities, such as,
 paper \cite{Ding-Lee} for super or asymptotically linear nonlinearities, papers \cite{Ackermann,TangxianhuaJDE} for nonlocal nonlinearities,
 and papers \cite{Chabrowski,Chen-Tang,weak-linking} for critical nonlinearities, etc.
 %%For more general results on this kind of problem,  we refer the reader to the papers  %%\cite{Bartsch-Ding,BC2,Chabrowski,Ding,Ding-multiple,SI1,Rup,weak-linking} and the references therein.
 However, %as in \cite{SK},
 all of these mentioned papers  require that the variational functional satisfies
 $\tau$-upper semi-continuity, which then implies that the nonlinear term in  (\ref{5.equation}) must have
 some kind of positivity, e.g., $F(x,u) \geq 0$.

%%In this paper,  our problem  (\ref{5.equation}) has a general sign-changing nonlinear term, that is, both $f(x,u)$ and its primitive function $F(x,u)$ change %%signs, so the methods developed in  \cite{SK} does not work in our case.

%% We mention that in \cite{SK} firstly
%% introduced the $\tau$-topology and   built the  generalized linking theorem which is a powerful tool to study the strongly indefinite problems
%% (see also chapter 6 of \cite{MW}).  They also improved the methods of \cite{V Rabinowitz2,V Rabinowitz1} and obtained the existence of infinitely
%%many geometrically distinct solutions of problem (\ref{5.equation}) under condition $(\textbf{V})$.
%%In \cite{Ding Lee} and  \cite{Ackermann} the authors studied the  existence of multiple solutions for problem (\ref{5.equation}) with asymptotically linear and %%nonlocal nonlinearities respectively.
%%For  more results on this problem  we   refer to the papers of \cite{Bartsch Ding,BC2,Chabrowski,Ding,Ding multiple,SI1,Rup,weak linking}.
%%However, most of the mentioned  results are based on the so-called $\tau$-upper semi-continuous assumption
%%(see section 2 for the definition)  which then require that the nonlinear term of
%%the equation has some kind of positivity.

Recently, the authors in \cite{Bernini-Bieganowski,CH} developed some critical point theorems
without assuming the $\tau$-upper semi-continuity and
 equations  with certain types of sign-changing nonlinear terms were studied (see also \cite{Chen-Chquard,FLiu-JianfuYang}).
More precisely,  in \cite{CH} the authors assumed further in equation  (\ref{5.equation}):
%let $\tilde{F}(x,u):=\frac{1}{2}uf(x,u)-F(x,u)$, and assume further:
\begin{description}
\item[$(\text{f}_{\text{3}})$] For some $q \in (2,2^*)$, let
\begin{equation*}
    \kappa := \max\{ \sup\limits_{u \in X\backslash\{0\}} \frac{\|Pu\|_{L^q}}{\|u\|_{L^q}}  , \sup\limits_{u \in X\backslash\{0\}} \frac{\|Qu\|_{L^q}}{\|u\|_{L^q}} \}
 \quad  \text{and}   \quad
\mu_0 := \inf\limits_{u \in X\backslash\{0\}} \frac{\|u\|^2}{\|u\|_{L^2}^2}.
\end{equation*}
There exist positive constants $\rho$, $D_1$ and $D_2$ such that
\begin{equation*}
    F(x,t)\geq 0, \  \tilde{F}(x,t) :=\frac{1}{2}uf(x,t)-F(x,t) \geq D_1|t|^q, \text{ and } |f(x,t)|\leq D_2 |t|^{q-1},
\end{equation*}
for all $x\in \mathbb{R}^N$ and $|t|\geq\rho$. In additional,
\begin{equation*}
    2\kappa D_2 \Big(    \rho^{q-2}+\frac{1}{D_1} \sup \limits_{|t|\leq \rho, x\in\mathbb{R}^N} \big|  \frac{\tilde{F}(x,t)}{t^2}  \big|         \Big)
    +\sup \limits_{|t|\leq \rho,x\in\mathbb{R}^N} \big|   \frac{f(x,t)}{t}   \big| <\mu_0.
\end{equation*}
\end{description}
Then, they obtained in \cite{CH} a nontrivial solution for equation (\ref{5.equation}) under conditions $(\textbf{V})$ and $\mathbf{(f_1)}$-$\mathbf{(f_3)}$.

\begin{remark}
{\rm i):} Condition $\mathbf{(f_3)}$ is mainly used to ensure the boundedness of (PS)-sequence, or (C)-sequence.
{\rm ii):} There are many functions satisfying $\mathbf{(f_1)}$-$\mathbf{(f_3)}$, for example,
if $2<r<q<2^*$, then
\begin{equation}\label{an example of nonlinear term}
    f(x,t)=|t|^{q-2} t - \lambda|t|^{r-2} t,
\end{equation}
satisfies conditions $\mathbf{(f_1)}$-$\mathbf{(f_3)}$ if  $\lambda>0$ is sufficiently small (see \cite[Remark 1.4]{CH}).
In fact,  for $f$ given by (\ref{an example of nonlinear term}), noting
\begin{equation*}
    \tilde{F}(x,t)=\frac{q-2}{2q}|t|^q-\lambda \frac{r-2}{2r}|t|^r=\frac{q-2}{4q}|t|^q+(\frac{q-2}{4q}|t|^q-\lambda \frac{r-2}{2r}|t|^r),
\end{equation*}
we take
$D_1=\frac{q-2}{4q}$ and $\rho_1 :=\Big(\frac{ 2q(r-2) }{r(q-2)} \Big)^{\frac{1}{q-r}} \lambda^{\frac{1}{q-r}}$,
 a direct computation shows that $\tilde{F}(x,t)=D_1|t|^q$ if $t=\rho_1$ and then $\tilde{F}(x,t)\geq D_1|t|^q$ if $t\geq\rho_1$.
By $|f(x,t)|\leq|t|^{q-1}   + \lambda|t|^{r-1} $, it is easy to see $|f(x,t)|\leq D_2 |t|^{q-1} $ if $t\geq\rho_1$, where
$D_2=1+\frac{r(q-2)}{2q(r-2)}$.
On the other hand we have $F(x,t)\geq 0$ if $t\geq \rho_2 :=  (\frac{q}{r})^{\frac{1}{q-r}}  \lambda^{\frac{1}{q-r}}$.
Clearly, by taking $\rho  =\max \{\rho_1 , \rho_2\}$ and due to $\rho$ can be arbitrarily small if $\lambda>0$ is small enough,
condition $\mathbf{(f_3)}$ is satisfied if $\lambda>0$ is small enough.
\end{remark}

The  sign-changing nonlinearity in  (\ref{5.equation}) plays an important role in nonlinear optics with material mixture of focusing and defocusing (see \cite{BieganowskiSurvey,Kuchment,DLMills,NieW,Pankov}).
For the positive definite case with nonlinearity difference of
two functions we refer to\cite{Bieganowski, Bieganowski-Mederski}.

To the best of our knowledge, there seems no any results about infinitely many solutions for  equation \eqref{5.equation} involving   sign-changing potential and nonlinearity. In this case,     the associated variational functional $\varphi(u)$ is not only strongly indefinite  but also $F$ in
$\varphi$ may change sign, and this makes it more challenging for finding multiple solutions.
In \cite{SK}, to get   infinitely many geometrically distinct solutions of (\ref{5.equation}) under the (AR) condition, the authors defined a special energy level
%\begin{equation*}
\[
     \beta=\max\limits_{u \in \mathcal{K}\backslash \{0\}} \varphi(u),
     \]
%\end{equation*}
where $\mathcal{K}:=\{u\in H^{1}(\mathbb{R}^{N}):\varphi'(u)=0 \}$ is assumed to be finite set under the action of $\mathbf{Z}^N-$translation,
and they first proved two kinds of deformation lemmas for general level sets under different energy levels,
 i.e.,  for $\varphi^{d+\epsilon}$ if $d\geq\beta+1$
and for $\varphi^{d+\epsilon}\backslash\mathcal{N}$ if $d<\beta+1$,  respectively, where $\mathcal{N}$ is a symmetric $\tau$-open set with
genus $\gamma(\bar{\mathcal{N}})=1$ (see \cite[Lemma 4.6]{SK}).
Then, they proved the existence of infinitely many geometrically distinct solutions of (\ref{5.equation}) by an indirect argument
originated by \cite{V-Rabinowitz1,V-Rabinowitz2}.
However, because the deformation lemmas  in \cite{SK} strictly rely on the $\tau$-upper semi-continuity,
their method seems invalid when $F$  changes sign. To overcome this difficulty,  in this paper we introduce a new energy level
\begin{equation*}
    \zeta:=\sup\limits_{\|Qu\| \leq    \delta+1}\varphi (u),
\end{equation*}
where $\delta:=\sup\{\|Qu\| :u\in \mathcal{K}\}$,
which is crucial to the proof of deformation lemmas respectively  for   two  symmetric bounded $\tau$-compact sets, i.e., for  $M$
 if $ \sup\limits_{u\in M}\varphi(u) >\zeta+1$, and for
$M\backslash\mathcal{N}$ (by a descending flow on $ \varphi^{\zeta+2}$) if $ \sup\limits_{u\in M}\varphi(u) <\zeta+2$
 (see  Lemmas \ref{5.first deformation lemma} and \ref{5.second deformation lemma} in Section 3).
To prove these  deformation lemmas,
some new ideas must be used in constructing  pseudo-gradient vector fields to avoid using the $\tau$-upper semi-continuity.
Furthermore, we also need to use a new strategy in  proving the split lemma of (PS)-sequence (see Lemma \ref{5.PS structure})
for the lack of  (AR) condition.

Before giving our main result of the paper, we introduce two further conditions on $f$ as follows:
\begin{description}
\item[$(\text{f}_{\text{4}})$] $f(x,-t)=-f(x,t) $ for all $x\in \mathbb{R}^N$ and $t\in \mathbb{R}$.
\item[$(\text{f}_{\text{5}})$]  There exist $\bar{c}>0$ and $ \varepsilon_0>0$ such that
\begin{equation*}
    |f(x,t+s)-f(x, t)|\leq \bar{c}|s|(1+|t|^{p-1})
\end{equation*}
for all $x\in \mathbb{R}^N$ and $t,s\in \mathbb{R}$ with $|s|\leq\varepsilon_0$.
\end{description}
Clearly, $f$ given by (\ref{an example of nonlinear term})  satisfies all the conditions $\mathbf{(f_1)}$-$\mathbf{(f_5)}$ if  $\lambda>0$ is sufficiently small.

Finally, we state our main result as follows:
\begin{theorem}\label{5.main result}
If the conditions $(\mathbf{V})$ and $\mathbf{(f_1)}$-$\mathbf{(f_5)}$ are satisfied, then
equation (\ref{5.equation}) has infinitely many geometrically distinct solutions.
\end{theorem}

%We mention  that  the  main difficulties for problem (\ref{5.equation}) with with  sign-changing potential and nonlinearity lie in the following two aspects: First, the  associated  variational functional $\varphi(u)$ lost the $\tau$-upper semicontinuity when $F(x,u)$ is sign-changing. Second, the nontrivial critical values of $\varphi(u)$ might not be strictly positive and then the method used in \cite[proposition 4.2]{SK} fails here.

%The difficulties above prevent us from using the  variational methods for strongly indefinite problems by a  traditional way. In order to  study the multiplicity of solutions to problem (\ref{5.equation}), we make the following innovations: First, new ways of  constructing the  pseudogradient vector fields were used to  avoid the use of $\tau$-upper semicontinuity. Second, we get the deformation lemmas for bounded $\tau$-compact sets instead of general level sets, and defined the pseudoindexes on the related classes. Last but not  least, we used some new method to analyze the structure of the (PS)-sequence for  $\varphi(u)$ and get  similar results as \cite[proposition 4.2]{SK} (see Lemma \ref{5.PS structure}).

Our theorem seems to be the first result on  multiple solutions for a periodic Schr\text{$\ddot{\text{o}}$}dinger
equation with both the potential and nonlinear term changing sign. In this case, we lose both the (AR) codition and the \textbf{$\tau$-upper semi-continuity} for the variational functional.
 %and the methods presented in this paper generalized the work \cite{BartschDingNon-Metrizable,SK} to a certain extent since the \textbf{$\tau$-upper semi-continuity is not needed here}.
%We think that these methods   may be used to investigate some more general variational problems with similar structure.
Moreover, the Fountain Theorem obtained in \cite{OUR}  does not work under the conditions of our
Theorem \ref{5.main result} because the (PS)-condition (which is needed in \cite{OUR})  cannot be satisfied when $\varphi$ is invariant under $\mathbf{Z}^N-$translation.

This paper is organized as follows: In section 2, some notations and several useful lemmas are  introduced.
In section 3,  two important deformation lemmas are established. In section 4,  Theorem \ref{5.main result} is proved
by the argument of contradiction based on a mini-max value defined by pseudo-indexes and the two deformation lemmas obtained in section 3.

\vskip4mm
 \par\noindent
\section{Preliminaries}
~

We begin this section by giving some notations and definitions. Let
\begin{equation}\label{defofhuaK}
     \mathcal{K}:=\{u\in H^{1}(\mathbb{R}^{N}):\varphi'(u)=0 \}
\end{equation}
be the set of   weak solutions of (\ref{5.equation}) and let
\begin{equation}\label{defofhuaF}
    \mathcal{F}:=\mathcal{K}/\mathbf{Z}^N
\end{equation}
be  the set of arbitrarily chosen representatives of $\mathcal{K}$ under the action of $\mathbf{Z}^N-$ translation.
By  \cite[Theorem 1.5]{CH},  we know that $\mathcal{F}\backslash\{0\}\neq\emptyset$ if conditions $\mathbf{(V)}$ and $\mathbf{(f_1)}$-$\mathbf{(f_3)}$ hold.
\begin{definition}\label{defofCcsequence}
Let  $\varphi\in C^{1}(X,\mathbb{R})$ and $c\in\mathbb{R}$. A sequence $\{u_{n}\}\subset X$  is called a $(C)^{c}$-sequence if
 $\sup\limits_{n}\varphi(u_{n})\leq c$  and   $(1+\|u_{n}\|)\|\varphi'(u_{n})\|_{X'}\rightarrow 0$   as   $n\rightarrow\infty$.
\end{definition}
\begin{definition}
Let  $\varphi\in C^{1}(X,\mathbb{R})$ and $c\in\mathbb{R}$. A sequence $\{u_{n}\}\subset X$  is called a $(PS)^{c}$-sequence if
 $\sup\limits_{n}\varphi(u_{n})\leq c$ and $\|\varphi'(u_{n})\|_{X'}\rightarrow 0$ as  $n\rightarrow\infty$.
\end{definition}

Clearly, a $(C)^{c}$-sequence is  a  $(PS)^{c}$-sequence, but only  a bounded $(PS)^{c}$-sequence is also a $(C)^{c}$-sequence.
As in \cite[section 2]{SK},  we have the following definition related to the
 $\tau$-topology.

 %\vspace{-2mm}
\begin{definition}
Let  $\{e_{j} \}_{j\geq 1}$   be an  orthonormal  basis of  $X^-$,
 the $\tau$-topology on $X$ = $X^- \oplus X^+$ is
 the topology associated with the following norm
\begin{center}
    $\| u \|_{\tau}:=
      \max\{\sum_{j=1}^{\infty}\frac{1}{2^{j}}|\langle Pu,e_{j}\rangle|,
      \| Qu \|\}$ ,  $u\in X$,
\end{center}
where $P$ and $Q$ are given in (\ref{defoforthogonalprojections}).
\end{definition}
By \cite[Remark 2.1(iii)]{SK},  we know that if $\{u_{n}\}\subset X$ is a bounded sequence, then
\vspace{-5mm}
 \begin{equation}\label{tau-converge dengjia}
u_{n} \stackrel{n}\rightarrow u \text{ in } \tau \text{-topology  } \Leftrightarrow
  Pu_{n} \stackrel{n}\rightharpoonup Pu  \text{ weakly in } X^-  \text{ and } Qu_{n} \stackrel{n}\rightarrow Qu \text{ strongly in } X^+.
 \end{equation}

Following the paper \cite{Rup}, we give now the  definitions
 on admissible homotopy/map.
\begin{definition}
Let $A\subset X$ be a closed subset. For $T>0$, a map $g : [0,T] \times A\rightarrow X$ is an admissible homotopy if
\begin{itemize}
  \item $g$ is $\tau$-continuous in the sense that, for $\{t_m\}\subset [0,T], \{u_m\}\subset A$,
             $$g(t_{m},u_{m})\overset{m}{\rightarrow} g(t,u)  \text{ in }   \tau\text{-topology} $$
             whenever $t_{m}\overset{m}{\rightarrow} t$ and $u_{m}\overset{m}{\rightarrow} u$ in $\tau$-topology.
  \item  For any $(t,u)\in [0,T] \times A$, there exists a neighborhood $W_{(t,u)}$ of $(t,u)$ in
          the $|\cdot|\times\tau-$ topology such that
          $$\{  v-g(s,v): (s,v)\in  W_{(t,u)}\bigcap  ([0,T] \times A  )       \}$$
          is contained in a finite-dimensional subspace of $X$.
\end{itemize}
\end{definition}
In particular, for a  map being independent of $t$, e.g., $g(t,u) \equiv g(u)$,  we have the following definition.
%\vspace{-5mm}
\begin{definition}
Let $A\subset X$ be a closed subset. A map  $h :  A\rightarrow X$ is an admissible map if
 $h$ is $\tau$-continuous  and, for any $ u \in  A$,  there is a $\tau$-neighborhood $W_u $ of $u$ such that $\{v-h( v): v \in  W_ u \cap A )\}$ 
    is contained in a finite-dimensional subspace of $X$.
\end{definition}

For  any integer $k\geq 1$, let
\[
S_{r}:=\{u\in X: \|u\|=r\}
\text{  and  }
X_{k}:=\overline{(\bigoplus_{j=1}^{k}\mathbb{R}f_{j})\oplus X^-},
\]
where   $\{f_{j} \}_{j\geq 1}$  is an  orthonormal  basis of   $X^+$. Then, we have
\begin{lemma}\label{5.lemma2.1}
Under the conditions $(\textbf{\emph{V}})$ and  $(\textbf{\emph{f}}_\textbf{\emph{1}})$-$(\textbf{\emph{f}}_\textbf{\emph{3}})$, let
$\varphi$ be  the functional defined by (\ref{5.energy1}),  then
\begin{description}
\item[(i)] There exists    $r>0$ such that
$$b:=\inf\limits_{u\in X^+\bigcap S_r} \varphi(u)>0.$$
\item[(ii)] For $r$ being obtained in  \textbf{\emph{(i)}}  and  any integer $k\geq 1$, there exists $R_k>r$ such that
$$\sup\limits_{u\in X_k,\|u\|=R_k} \varphi(u)   <  a :=\inf\limits_{\|u\|\leq r} \varphi(u).$$
\item[(iii)] For any $\delta<+\infty$, there holds
$$\sup\limits_{\|Qu\| \leq \delta} \varphi(u)< +\infty.$$
Moreover,
$$\limsup\limits_{\|Qu\| \rightarrow0} \varphi(u)\leq 0.$$
\end{description}
\end{lemma}
\begin{proof}
\textbf{(i)} and \textbf{(iii)} directly follows from the steps 1 and  3 in the proof of  \cite[Lemma 3.1]{CH},  respectively.
Similar to the step 2 of the proof of  \cite[Lemma 3.1]{CH}, we know that $\varphi(u)\rightarrow -\infty$
as $\|u\|\rightarrow -\infty$ and $u \in X_k$, hence part \textbf{(ii)}  is also  proved.
\end{proof}

Clearly, $a :=\inf\limits_{\|u\|\leq r} \varphi(u)\leq b:=\inf\limits_{u\in X^+\bigcap S_r} \varphi(u)$.
By Lemma \ref{5.lemma2.1}\textbf{(i)} and  \textbf{(iii)},
there exist   $r>0$  and $\delta_0 >0$  such that
\begin{equation}\label{5.r}
 \sup\limits_{\|Qu\|\leq \delta_0} \varphi(u)  < b:= \inf\limits_{u\in X^+\bigcap S_r} \varphi(u).
\end{equation}

By (3.18)  of  \cite{CH}, we know that, if  $\mathbf{(V)}$ and  $\mathbf{(f_1)} $-$\mathbf{(f_3)}$ hold,  there exist constants $c_1>0$ and $c_2>0$, which depend only on $\mu_0$ in $\mathbf{(f_3)}$, such that
\begin{equation}\label{xiajie to youjie}
    \varphi(u)\leq c_1 \|Qu\|^2 - c_2\|Pu\|^2.
\end{equation}
Hence,  (\ref{tau-converge dengjia}) and (\ref{xiajie to youjie}) imply that
\begin{lemma}\label{5.upper bounded}
Under  conditions $(\textbf{\emph{V}})$ and  $(\textbf{\emph{f}}_\textbf{\emph{1}})$-$(\textbf{\emph{f}}_\textbf{\emph{3}})$, let
$\{u_n\} \subset \varphi_c:=\{u\in X : \varphi(u)\geq c\}$ be any sequence with   $\|u_n-u_0\|_\tau \overset{n}{\rightarrow} 0$,  then
\begin{equation*}
    Pu_n \overset{n}{\rightharpoonup} Pu_0  \text{ weakly in } X^- \text{ and  } Qu_n\overset{n}{\rightarrow} Qu_0 \text{ strongly in } X^+.
\eqno{\Box}
\end{equation*}
 %\hfill{$\Box$}
\end{lemma}

In \cite[Lemma 3.2]{CH}, it is  proved that any  $(C)^{c}$-sequence of  $\varphi$, see  (\ref{5.energy1}),  is bounded in $X$ under conditions $(\textbf{\textbf{V}})$ and  $(\textbf{\textbf{f}}_\mathbf{\textbf{1}})$-$(\textbf{\textbf{f}}_\textbf{\textbf{3}})$.
%%In their work using a $(C)^{c}$-argument
%% is easier when judging the weak
%%limit of a $(C)^{c}$-sequence is nontrivial or not. But in our work, the  $(PS)^{c}$-argument is much more appropriate when proving the deformation lemmas
%%in section 3.
By an argument almost the same as   \cite[Lemma 3.2]{CH}, we can easily prove that
\begin{lemma}\label{5.PS bounded}
Let  $(\textbf{\emph{V}})$ and  $(\textbf{\emph{f}}_\textbf{\emph{1}})$-$(\textbf{\emph{f}}_\textbf{\emph{3}})$ be satisfied. If
$\{u_n\}\subset X$ is a $(PS)^{c}$-sequence of $\varphi$, then
$$\sup\limits_{n}\|u_n\|\leq C,$$
for some $C>0$ (independent of $n$).
\hfill{$\Box$}
\end{lemma}

So, if $\{u_n\}$ is a $(PS)^{c}$-sequence of $\varphi$, by Lemma \ref{5.PS bounded}  we may  define
$$ M_c:= \limsup \limits_{n\rightarrow\infty} \|u_n\|,$$
and $M_c\leq C <+\infty$.
Moreover,  we have the following compactness result.

\begin{lemma}\label{5.PS structure}
Under  conditions $(\textbf{\emph{V}})$ and  $(\textbf{\emph{f}}_\textbf{\emph{1}})$-$(\textbf{\emph{f}}_\textbf{\emph{5}})$,
let $\{u_m\}\subset X$ be a $(PS)^{c}$-sequence for (\ref{5.energy1}), $\mathcal{K}$ and $\mathcal{F}$ are defined by
(\ref{defofhuaK}) and (\ref{defofhuaF}) respectively.
 If $\alpha:=\inf\limits_{u\in \mathcal{K}\backslash\{0\}} \|u\|>0$, then either\\

  \textbf{\emph{(i)}} $\limsup \limits_{m\rightarrow\infty} \|u_m\|=0,$ or, \\

 $\textbf{\emph{(ii)}}$ there exist a positive integer $l\leq [\frac{M^2_c}{\alpha^2}]$, points $\bar{u_1},...,\bar{u_l} \in \mathcal{F}\backslash \{0\}$
 ($\bar{u_i}, i=1,...,l$, are not necessarily distinct),
  a subsequence of $\{u_m\}$ (still denoted by $\{u_m\}$) and sequences $\{g_{m}^{i}\}\subset \mathbf{Z}^N$, ($i=1,...,l$) such that
\begin{equation*}
   \lim\limits_{m\rightarrow \infty}  \|u_m - \sum \limits_{i=1}^{l}(g^{i}_{m}\ast \bar{u}_{i})  \|=0,
\end{equation*}
where $[\cdot]$ denotes the integer part of a real number and $(g \ast u)(x):=u(x+g)$ for $g \in \mathbf{Z}^N$.
\end{lemma}
\begin{proof}
For $u\in X$ we will denote $u^+:=Qu$ and $u^-:=Pu$ for convenience.
By contradiction, if  $\textbf{(i)}$ is not satisfied,
by Lemma \ref{5.PS bounded},  $\{u_m\}$ is also a $(C)^{c}$-sequence, then
  similar to the proof of Theorem 1.5 in \cite{CH} (see also Lemma 1.7 in \cite{SK})  we know  that there exist a weak convergent
 subsequence of $\{u_m\}$ (still denoted by $\{u_m\}$), a sequence $\{a_m\}\subset\mathbb{R}^N$ and two constants $r$, $\eta>0$ such that
\begin{equation*}
    \|u_m\|_{L^{2}(B(a_{m},r))}> \eta,
\end{equation*}
for all $m\in \mathbf{N}$, where $B(a_m,r)=\{x\in \mathbb{R}^N:|x-a_m|\leq r\}$.
We may choose $\{g_{m}\}\subset \mathbf{Z}^N$ and set $v_m:=(g_{m}\ast u_{m})$ such that
\begin{equation*}
     \|v_m\|_{L^{2}(B(0,r+\frac{\sqrt{N}}{2}))}> \eta, \text{ for all } m\in \mathbf{N}.
\end{equation*}
It is easy to see that $\|v_m\|=\|u_m\|$, so, $\{v_{m}\}$ is also bounded in $X$ and there is a subsequence (still denoted by $\{v_{m}\}$) converges to some $v\in X$ both weakly in $X$ and strongly in $L^{s}_{loc}(\mathbb{R}^N)$ for $s\in [2,2^*)$. Therefore
\begin{equation*}
    v\in \mathcal{K}\backslash \{0\}.
\end{equation*}
Let $w_m:=v_m-v$, we claim that
\begin{equation}\label{5.2.1}
   \lim\limits_{m\rightarrow\infty} \varphi'(w_m)= 0,
\end{equation}
and
\begin{equation}\label{5.2.2}
    \limsup\limits_{m\rightarrow\infty}\|w_{m}\|^{2} \leq M_{c}^2-\|v\|^{2}.
\end{equation}
%%where $\|\cdot\|$ is defined by (\ref{defofnewnorm2}).
We  prove (\ref{5.2.1}) first.  Let $\langle\cdot,\cdot\rangle$ be the inner product defined by (\ref{defofnewnorm}),
it follows from (\ref{5.derivative2}) that, for any $h\in X$ with $\|h\|=1$,
  \begin{equation}\label{5.2.3}
    \varphi'(w_m) h =\langle (v^{+}_{m}-v^{-}_{m}),h \rangle-\langle (v^+ - v^-),h \rangle-\int_{\mathbb{R}^{N}}f(x,w_m)h(x) dx.
  \end{equation}
Since $ \varphi'(v)=0$ and  $\varphi'(v_m) \overset{m}{\rightarrow} 0 $ in $X'$, these imply that
  \begin{equation}\label{5.2.4}
    \langle (v^+ - v^-),h \rangle=\int_{\mathbb{R}^{N}}f(x,v)h(x) dx,
  \end{equation}
  and
  \begin{equation}\label{5.2.5}
  \langle (v_m^+ - v_m^-),h \rangle=\int_{\mathbb{R}^{N}}f(x,v_m)h(x) dx + o(1), \quad \text{as} \quad m\rightarrow\infty.
  \end{equation}
  Then, (\ref{5.2.3}) together with (\ref{5.2.4}) and (\ref{5.2.5}) gives that
  \begin{equation*}
     \varphi'(w_m)h =\int_{\mathbb{R}^{N}}\big(f(x,v_m)-f(x,w_m)-f(x,v)\big)h(x) dx + o(1), \quad \text{as} \quad m\rightarrow\infty.
  \end{equation*}
 So, we only need to show that
\begin{equation*}
   \int_{\mathbb{R}^{N}} |\big(f(x,v_m)-f(x,w_m)-f(x,v)\big)h(x)| dx \overset{m}{\rightarrow} 0, \text{ for } \|h\|=1 \text{ uniformly}.
\end{equation*}

Since $v$ is a solution of (\ref{5.equation}), we have $-\Delta v +c(x)v=0$, where $c(x)=V(x)-\frac{f(x,v)}{v}$. It follows from $(\mathbf{f}_\mathbf{1})$
and $(\mathbf{f}_\mathbf{2})$ that there exists a $\delta>0$ such that
 \begin{equation}\label{5.2.6}
    |f(x,u)|\leq \delta|u|+c_{\delta}|u|^{p-1},
 \end{equation}
 thus $c(x) \in L_{loc}^{t}({\mathbb{R}^N})$ for some $t>\frac{N}{2}$. By Theorem 4.1 of \cite{Hanqing},
 we have $v(x)\rightarrow 0$ as $|x|\rightarrow\infty$.
Then, for any $ \varepsilon >0$ and $R>0$, $B_R=\{x \in \mathbb{R}^N:|x|<R\}$ with $R$ large enough,
by $(\mathbf{f}_\mathbf{5})$ and (\ref{5.2.6}) there holds
\begin{eqnarray}
 \nonumber  & & \int_{\mathbb{R}^{N}\backslash B_R} | \big(f(x,v_m)-f(x,w_m)-f(x,v)\big)h(x) |dx \\
 \nonumber  &\leq &  \int_{\mathbb{R}^{N}\backslash B_R} \bar{c}|v|(1+|w_m|^{p-1})h +C \int_{\mathbb{R}^{N}\backslash B_R} |v||h|\\
   &\leq &  C|v|^2 \|h\|^2+C|v| \|w_m\| ^{p-1} \|h\|    \leq\frac{\varepsilon}{2}. \label{5.2.7}
\end{eqnarray}
 Moreover, since
both $ w_m \overset{m}{\rightarrow} 0 $ in $L^p(B_R)$ and $ v_m \overset{m}{\rightarrow} v$ in $L^p(B_R)$, it follows from
 Theorem A.2 of \cite{MW} that
\begin{eqnarray}
\nonumber   & &  \int_{B_R} | \big(f(x,v_m)-f(x,w_m)-f(x,v)\big)h(x) |dx \\
   &\leq &    \int_{B_R} | \big(f(x,v_m)-f(x,v)\big)||h(x) |dx  + \int_{B_R} |f(x,w_m)||h(x) | \leq\frac{\varepsilon}{2}, \label{5.2.8}
\end{eqnarray}
for $m$ large enough. Since $\varepsilon$ is arbitrary,   (\ref{5.2.7}) and (\ref{5.2.8}) show that
\begin{equation}\label{5.2.9}
 \int_{\mathbb{R}^{N}} |\big(f(x,v_m)-f(x,w_m)-f(x,v)\big)h(x)| dx \overset{m}{\rightarrow} 0,
\end{equation}
uniformly for $\|h\|=1$. Hence (\ref{5.2.1}) is proved.

 Next we  prove (\ref{5.2.2}). First from
 $\varphi'(v)=0 $, $\varphi'(v_m)\overset{m}{\longrightarrow}0 $, $\varphi'(w_m)\overset{m}{\longrightarrow}0  $ in $X'$ and the boundedness of
 $\|v_m\| $ and  $ \|w_m\|$ we have
 $$ \|v\|^2= \int_{\mathbb{R}^{N}} f(x,v)(v^+- v^-), $$
 $$ \|v_m\|^2= \int_{\mathbb{R}^{N}} f(x,v_m)(v_m^+- v_m^-)+o(1), $$
 $$ \|w_m\|^2= \int_{\mathbb{R}^{N}} f(x,w_m)(w_m^+- w_m^-)+o(1). $$
So,
\begin{eqnarray}
  \nonumber  & & \|w_m\|^2+ \|v\|^2-\|v_m\|^2\\
 \nonumber  &= &\int_{\mathbb{R}^{N}} f(x,w_m)(v_m^+- v_m^-)-\int_{\mathbb{R}^{N}} f(x,w_m)(v ^+- v ^-)\\
& & +\int_{\mathbb{R}^{N}} f(x,v)(v^+- v^-)-\int_{\mathbb{R}^{N}} f(x,v_m)(v_m^+- v_m^-)+o(1). \label{5.2.10}
\end{eqnarray}
Since  $v_m \overset{m}{\rightharpoonup}v$ weakly in $X$, we have $w_m \overset{m}{\rightharpoonup}0$  weakly in $X$ and thus
\begin{equation}\label{5.2.11}
    \langle w_m, v^{\pm}\rangle   \rightarrow  0 \quad \text{as} \quad m\rightarrow\infty.
\end{equation}
By $\langle \varphi'( w_m) , v^{\pm}\rangle \overset{m}{\longrightarrow}0$, we have
\begin{equation*}\label{5.2.12}
     \langle w_m, v^+\rangle=\int_{\mathbb{R}^{N}} f(x,w_m) v^+  +o(1),
\end{equation*}
and
\begin{equation*}\label{5.2.13}
      -\langle w_m, v^-\rangle=\int_{\mathbb{R}^{N}} f(x,w_m) v^-  +o(1).
\end{equation*}
These together with (\ref{5.2.11}) imply that
\begin{equation}\label{5.2.14}
     \lim\limits_{m\rightarrow\infty}  \int_{\mathbb{R}^{N}} f(x,w_m) v^{\pm}  =0.
\end{equation}
By the boundedness of $\|w_m\|$, we have $\varphi'(v)w_m^{\pm} =0$, so,
\begin{equation}\label{5.2.15}
    \langle v, w_m^+ \rangle= \int_{\mathbb{R}^{N}} f(x,v) w_m^{+} \quad\text{ and } \quad -\langle v, w_m^- \rangle= \int_{\mathbb{R}^{N}} f(x,v) w_m^{-}.
\end{equation}
Since   $w_m \overset{m}{\rightharpoonup}0$ weakly in $X$,  then $w_m^{\pm} \overset{m}{\rightharpoonup}0$ weakly in $X^{\pm}$, and (\ref{5.2.15}) shows that
\begin{equation*}
    \int_{\mathbb{R}^{N}} f(x,v) w_m^{+} = \int_{\mathbb{R}^{N}} f(x,v) (v_m^{+}-v^+)\rightarrow 0 \quad \text{as} \quad m\rightarrow\infty,
\end{equation*}
and
\begin{equation*}
  \int_{\mathbb{R}^{N}} f(x,v) w_m^{-}= \int_{\mathbb{R}^{N}} f(x,v) (v_m^{-}-v^-) \rightarrow 0\quad  \text{as} \quad  m\rightarrow\infty.
\end{equation*}
So,
\begin{equation}\label{5.2.16}
    \int_{\mathbb{R}^{N}} f(x,v)(v_m^+ - v_m^-)=  \int_{\mathbb{R}^{N}} f(x,v)(v ^+ - v ^-)+o(1).
\end{equation}
Then, it follows from (\ref{5.2.10}), (\ref{5.2.14}) and (\ref{5.2.16}) that
\begin{eqnarray*}
  \nonumber  & & \|w_m\|^2+ \|v\|^2-\|v_m\|^2\\
         &= &\int_{\mathbb{R}^{N}} \Big(f(x,w_m)+f(x,v)-f(x,v_m)\Big)(v_m^+- v_m^-) +o(1)
\end{eqnarray*}
By the boundedness of $\{\|v_m\|\}$  and using (\ref{5.2.9}) again, we have
\begin{equation*}
    \lim\limits_{m\rightarrow\infty} \{\|w_m\|^2+ \|v\|^2-\|v_m\|^2\}=0,
\end{equation*}
which leads to (\ref{5.2.2}) by  using the definition of $M_c$.

Since $\|v\|\leq M_c$ (by the weak lower semicontinuity of norms), there are two possibilities which may occur:
\begin{itemize}
  \item If $\|v\|=M_c$, then  (\ref{5.2.2}) implies that $w_m\overset{m}{\rightarrow} 0$ strongly in $X$, that is  $\mathbf{(ii)}$ holds with $l=1$
               and $\bar{u}_{1}=v$.
  \item If $\|v\|<M_c$,  we may go back to the beginning of our proof by simply replacing   $\{u_m\}$ and $M_c^2$   by $\{w_m\}$ and $M_c^2-\|v\|^2$ respectively.
   Then we can set $\bar{u}_{2}\in \mathcal{K}$ with $\alpha^2  \leq \|\bar{u}_{2}\|^2\leq M_c^2-\alpha^2$.
    Repeat this procedure    at most $[\frac{M_c^2}{\alpha^2}]$ times, we obtain the conclusion.
\end{itemize}
\end{proof}

Let  $l\in \mathbf{N}$ and $\mathcal{A}\subset X$ be a finite set, i.e., $\mathcal{A}$ contains finite number  of elements, define
\begin{equation*}
    [\mathcal{A},l]:=\Big\{  \sum \limits_{i=1}^{j}(g_{i}\ast  v _{i}) :    1\leq j\leq l, g_{i} \in \mathbf{Z}^{N}, v_{i} \in \mathcal{A}  \Big\}.
\end{equation*}
Then, it  follows   from \cite{V-Rabinowitz2}, see also  \cite{V-Rabinowitz1, SK}, that

\begin{lemma}\label{5.lemma2.5}
\rm{(\cite{V-Rabinowitz2}, Proposition 2.57)} For any $l \in \mathbf{N}$, if  $\mathcal{A}\subset X$ is  a finite set, then
\begin{equation*}
     \inf \{ \|v-v'\|: v, v' \in [\mathcal{A},l], v\neq v'    \}>0.
\end{equation*}
\hfill{$\Box$}
\end{lemma}

We mention that the sets $[\mathcal{A},l]$ defined above are closed related to the so-called $(PS)$-attractors or  $(C)$-attractors
(see \cite{BartschDingNon-Metrizable}).
By Lemmas \ref{5.PS structure}, for $\mathcal{F}$  defined by (\ref{defofhuaF}) and $\varphi$ defined by (\ref{5.energy2}),
 we have the following lemma

\begin{lemma}\label{5.lemma2.6}
Under  conditions $(\textbf{\emph{V}})$ and  $(\textbf{\emph{f}}_\textbf{\emph{1}})$-$(\textbf{\emph{f}}_\textbf{\emph{5}})$,
if $\mathcal{F}$  is a finite set and
 $\{u_m\}\subset X $ is a $(PS)^c$-sequence of $\varphi$, then there exists $l_{c} \in \mathbf{N}$, which depends only on $c$, such that
\begin{equation*}
    0\leq \|u_m  - [\mathcal{F},l_c] \|_{\tau} \leq \|u_m  - [\mathcal{F},l_c] \| \rightarrow 0 \quad \text{as}  \quad   m\rightarrow \infty.
\end{equation*}
\hfill{$\Box$}
\end{lemma}

Let
\begin{equation*}
    B_{X^+}(z,r):= B(z,r)\cap X^+,~ B(z,r):=\{u\in X : \|u-z\|<r\},
\end{equation*}
then it follows from Lemmas \ref{5.PS structure}-\ref{5.lemma2.6} that

\begin{lemma}\label{5.lemma 2.7}
If conditions $(\textbf{\emph{V}})$ and  $(\textbf{\emph{f}}_\textbf{\emph{1}})$-$(\textbf{\emph{f}}_\textbf{\emph{5}})$ hold,
$\mathcal{F}$ is a finite set,
 then, for any $c\in \mathbb{R}$ and
$s \in (0,\frac{\mu}{4})$ with
\begin{equation*}
   \mu= \inf\{\|v-v'\|:v,v' \in [Q\mathcal{F},l_c], v\neq v'\},
\end{equation*}
where $l_{c} \in \mathbf{N}$ is given in Lemma \ref{5.PS structure} and depends only on  $c$, there exists   $\sigma>0$    such that
\begin{equation*}
   \| \varphi'(u)\|_{X'} >\sigma   \text{ for any }   u \in \varphi^{c}\backslash \big( \bigcup\limits_{z \in [Q\mathcal{F},l_{c}] }  X^{-} \oplus B_{ X^{+} }(z,s) \big), ~\varphi^{c}:=\{u\in X: \varphi(u)\leq c\}.
\end{equation*}

\end{lemma}

\begin{proof}
By contradiction, if  such $\sigma$ does not exist, then there exists a sequence
\begin{equation}\label{5.lemma 2.7.1}
    \{u_m\}\subset \varphi^{c}\backslash \big( \bigcup\limits_{z \in [Q\mathcal{F},l_{c}] }  X^{-} \oplus B_{ X^{+} }(z,s) \big),
     \text{ for some } s \in (0,\frac{\mu}{4}),
\end{equation}
such that $  \varphi'(u_m)  \overset{m}{\rightarrow} 0$ in $X'$. Then,  Lemma \ref{5.lemma2.6} implies that
\begin{equation*}
    \|u_m  - [\mathcal{F},l_c] \| \rightarrow 0 \quad \text{as}  \quad   m\rightarrow \infty.
\end{equation*}
So,
\begin{equation*}
      \|Qu_m  - Q[\mathcal{F},l_c] \| \rightarrow 0 \quad \text{as}  \quad   m\rightarrow \infty.
\end{equation*}
On the other hand, by \cite[Remark 1.1 (iv)]{SK} $,  Q$ and  $\ast$ are commutable, thus
$$ Q[\mathcal{F},l_c]=[Q \mathcal{F},l_c].$$
So,
\begin{equation*}
      \|Qu_m  -[Q \mathcal{F},l_c] \| \rightarrow 0 \quad \text{as}  \quad   m\rightarrow \infty,
\end{equation*}
which contradicts (\ref{5.lemma 2.7.1}). The proof is complete.

\end{proof}

\section{Two deformation lemmas}
~

In this section, two deformation lemmas will be proved and we stress  that the $\tau$-upper semi-continuity of $\varphi$ is not needed.
Assume that (\ref{5.equation}) has only finite  geometrically distinct solutions, i.e., $\mathcal{F}$ is finite, then
\begin{equation}\label{5.3.1}
    \delta:=\sup\{\|Qu\| :u\in \mathcal{K}\}<+\infty.
\end{equation}
By Lemma \ref{5.lemma2.1} \textbf{(iii)}, we have
\begin{equation}\label{5.3.2}
    \zeta:=\sup\limits_{\|Qu\| \leq    \delta+1}\varphi (u) <+\infty.
\end{equation}
Let
$$\Sigma := \{A\subset X: A \text{ is closed and } A=-A\},$$
and
\begin{equation}\label{defofsigematil}
   \tilde{\Sigma} := \{A\in \Sigma: A \text{  is bounded and } \tau\text{-compact}\}.
\end{equation}
Since $X$ is a Hilbert space, let $\nabla\varphi$ be given by the formula
$$  \langle \nabla\varphi(u),v \rangle= \varphi'(u)v, \text{ for all } v \in X,  $$
then we have
\begin{lemma}\label{5.vectorfield1}
 Under conditions $(\textbf{\emph{V}})$, $(\textbf{\emph{f}}_\textbf{\emph{1}})$-$(\textbf{\emph{f}}_\textbf{\emph{5}})$, let
%%$\varphi$ is given by (\ref{5.energy1})
 $\mathcal{F}$ be a  finite set and let $M \in\tilde{\Sigma} $.  Denote
 \begin{equation*}
    \bar{\beta} := \sup\limits_{u\in M} \varphi(u) \text{ and }  \beta:= \inf\limits_{u\in M} \varphi(u).
 \end{equation*}
 If  $\bar{\beta} > \zeta+1$, $\zeta$  is defined in  (\ref{5.3.2}),
then there exists  a vector field:
$$\chi_1(u): \ \varphi^{\bar{\beta}}_{\beta-3}\rightarrow X \text{ with } \varphi^{\bar{\beta}}_{\beta-3}:=\{u\in X: \beta-3\leq\varphi(u)\leq \bar{\beta}\}$$
 such that
\begin{description}
  \item[(i)] $\chi_1(u)$ is locally Lipschitz continuous   and $\tau$-locally Lipschitz $\tau$-continuous on $\varphi^{\bar{\beta}}_{\beta-3}$.
  \item[(ii)] $\chi_1(u)$ is odd with    $-3\leq\langle\nabla\varphi(u),\chi_{1}(u)\rangle \leq 0$,  for  any $ u\in  \varphi^{\bar{\beta}}_{\beta-3}$.

  \item[(iii)] $\langle \nabla\varphi(u),\chi_1(u)\rangle  <-1$  for any   $u\in E_1:= \varphi^{\bar{\beta}}_{\beta-3}\backslash \{u\in X:\|Qu\| <  \delta+1\}$, for $\delta$ defined by (\ref{5.3.1}).
\item[(iv)] There exists $\sigma_1>0$ such that
$$\|\chi_1(u)\|<\sigma_1 \text{ for any } u\in \varphi_{\beta-3}^{\bar{\beta}}\backslash \big( \bigcup\limits_{z \in [Q\mathcal{F},l_{\bar{\beta}}] }  X^{-} \oplus B_{ X^{+} }(z,\frac{\mu}{4}) \big)$$
where $\mu$ and $l_{\bar{\beta}}$ are given by Lemma \ref{5.lemma 2.7}.
\item[(v)] Each $u\in \varphi^{\bar{\beta}}_{\beta-3}$ has  a $\tau$-open neighborhood, $U_{u}\subset X$,
of $u$ such that   $\chi_1(U_{u}\cap\varphi^{\bar{\beta}}_{\beta-3})$ is contained in a finite-dimensional subspace of $X$.
%%\item[(v)] If $A$ is $\tau$-compact and $\varphi(A)$ is bounded from below, then
%%              \begin{equation} \label{5.3.3}
%%                 \sup\limits_{u\in A}\|\chi(u)\|<\infty.
%%             \end{equation}
\end{description}

\end{lemma}

\begin{proof}
For any $u\in E_1$, let
$$\omega(u)=\frac{2\nabla\varphi(u)}{\|\nabla\varphi(u)\|^{2}},$$
then, there exists a $\tau$-neighborhood of $u$, $V_{u}\subset X$, satisfying
\begin{equation}\label{5.lessthan mu 8}
     \|v-u\|_{\tau}<\text{min}\{\frac{\mu}{8},\frac{1}{3}\},  \text{ for any }  v\in V_{u},
\end{equation}
  such that
$$ 1< \langle \nabla\varphi(v),  \omega(u) \rangle<3,  \text{ for any }   v\in V_{u}\cap \varphi^{\bar{\beta}}_{\beta-3}\backslash \{u\in X:\|Qu\| <  \delta+\frac{1}{2}\}.$$
Indeed, suppose that such $V_{u}$ does not exist, then there exists a sequence $\{v_n\}\subset\varphi _{\beta-3}:=\{u\in X: \varphi(u)\geq\beta-3\}$
with $v_n \overset{\tau}{\rightarrow} u $ and $\lim\limits_{n\rightarrow\infty}\langle \nabla\varphi(v_n),  \omega(u) \rangle \leq 1$ or $\geq 3$.
By Lemma \ref{5.upper bounded}, $v_n \overset{n}{\rightharpoonup} u$ weakly in $X$, which leads to a contradiction
since $\nabla\varphi$ is weakly continuous.
So,

$$\mathcal{N}_1=\{V_{u}:u\in E_1 \} \cup \{u\in X:\|Qu\| <  \delta+1\}$$
forms a $\tau$-open covering of $\varphi^{\bar{\beta}}_{\beta-3}$.

Since $\mathcal{N}_1$ is metric and paracompact, that is, there exists a locally finite $\tau$-open covering
$\mathcal{M}_1=\{M_{i}:i\in \Lambda \}$ of $\varphi^{\bar{\beta}}_{\beta-3}$ which is  finer than $\mathcal{N}_1$, where $\Lambda$ is an index set.
If $M_{i}\subset V_{u_{i}}$
for some $u_{i}\in E_1$, we take $\omega_{i}=\omega(u_{i})$ and if $M_{i}\subset \{u\in X:\|Qu\| <  \delta+1\}$,
we take $\omega_{i}=0$. Let $\{\lambda_{i}:i \in \Lambda\}$ be a $\tau$-Lipschitz continuous partition
of unity subordinated to $\mathcal{M}_1$, then we define the following vector field in $\mathcal{N}_1$

$$\xi_{1}(u)=\sum_{i \in \Lambda}\lambda_{i}(u)\omega_{i}, \quad u\in \mathcal{N}_1.$$

Since the $\tau$-open covering $\mathcal{M}_1$ of $\mathcal{N}_1$ is locally finite, each $u\in\mathcal{N}_1$ belongs to finite many sets $M_{i},i\in \Lambda$, that is, for any $u\in \mathcal{N}_1$ the sum $\xi_1(u)$ has only   finite terms. It follows that, for any $u\in \mathcal{N}_1$, there exist a $\tau$-open neighborhood $U_{u} \in \mathcal{M}_1$ of $u$ and  a constant $L_{u}>0$ such that  $\xi_1(U_{u})$ is contained in a finite-dimensional subspace of $X$ and

$$\|\xi_1(v)-\xi_1(w)\|\leq L_{u} \|v-w\|_{\tau},  \text{ for any } v,w \in U_{u},$$
where we used the fact that all norms of a finite-dimensional vector space are equivalent. This gives that $\xi_1(u)$ is locally Lipschitz continuous and $\tau$-locally Lipschitz $\tau$-continuous, and we also have

   $$1<\langle \nabla\varphi(u),\xi_1(u)\rangle  <3  \text{ for any }   u\in E_1.$$
Define
 $$ \tilde{\xi}_1(u)=\frac{\xi_1(u)-\xi_1(-u)}{2},$$
and let $\theta\in C^{\infty}(\mathbb{R})$ with $0\leq \theta \leq1$   be such that

 $$  \theta(t)=
 \left\{
   \begin{array}{ll}
     0,  \quad t\leq  \delta+\frac{1}{2} ,\\
     1,   \quad t\geq   \delta+\frac{2}{3},
   \end{array}
   \right.$$
where  $\delta$ is given in (\ref{5.3.1}).
Now, we define the vector field $\chi_1$ as follows

$$  \chi_1(u)=
 \left\{
   \begin{array}{ll}
     -\theta(\|Qu\| ) \tilde{\xi}_1(u),  \quad u\in \mathcal{N}_1,\\
     0,   \quad \|Qu\|  \leq\delta+\frac{1}{2},
   \end{array}
   \right.$$
then, \textbf{(i)}-\textbf{(iii)}, \textbf{(v)} follow  directly from the construction of $\chi_1(u)$.
By Lemma \ref{5.lemma 2.7},  there exists $\sigma>0$ such that

$$ \|\nabla\varphi(u)\|>\sigma, \text{ for any } u \in\varphi^{\bar{\beta}}\backslash \big( \bigcup\limits_{z \in [Q\mathcal{F},l_{\bar{\beta}}] }  X^{-} 
\oplus B_{ X^{+} }(z,\frac{\mu}{8}) \big),$$
 then, $\|\omega(u)\|< \frac{2}{\sigma}:=\sigma_{1}$, for any $u \in \varphi_{\beta-3}^{\bar{\beta}}\backslash \big( \bigcup\limits_{z \in [Q\mathcal{F},l_{\bar{\beta}}] }  X^{-} \oplus   B_{ X^{+}}(z,\frac{\mu}{8})  \big)$. This together with (\ref{5.lessthan mu 8}) gives \textbf{(iv)}.
\end{proof}

\begin{definition}\label{5.def of H A}
For each $A \in \tilde{\Sigma}$,   define $\mathcal{H}(A)$ to be  the class of maps $h: A\rightarrow X$ satisfying
\begin{itemize}
  \item $h$ is a homeomorphism of $A$ onto $h(A)$ in the original topology, i.e., the topology induced by $\|\cdot\|$, of $X$;
  \item $h$ is an odd admissible map which maps bounded set into bounded set;
  \item $\varphi(h(u)) \leq \varphi(u)$, for any $u \in A$.
\end{itemize}
\end{definition}

\begin{remark}\label{5.remark3.1}
 For any $A \in \tilde{\Sigma}$, $\mathcal{H}(A)$ is nonempty since it always contains the identity. Furthermore, $\mathcal{H}(A)$ is closed under composition, i.e., for any $A \in \tilde{\Sigma}$ and $h\in \mathcal{H}(A)$, $g\in \mathcal{H}\big(h(A)\big)$,  there holds $g\circ h \in \mathcal{H}(A)$, since $h(A)$ is also bounded and $\tau$-compact for any $h\in \mathcal{H}(A)$.
\end{remark}

Now, we give our first deformation lemma.
%%Different from the deformation lemma of \cite{SK,Rup}, we can find admissible  homotopy  only
%%for  bounded  $\tau$-compact sets because of the lack of $\tau$-upper semicontinuous. But that is enough for our proof of Theorem \ref{5.main result}.
\begin{lemma}\label{5.first deformation lemma}{\bf (Deformation Lemma 1)}
Under conditions $(\textbf{\emph{V}})$, $(\textbf{\emph{f}}_\textbf{\emph{1}})$-$(\textbf{\emph{f}}_\textbf{\emph{5}})$, let
%%$\varphi$ is given by (\ref{5.energy1})
 $\mathcal{F}$ be a  finite set and let $M \in\tilde{\Sigma} $.  If  $\bar{\beta}:= \sup\limits_{u\in M} \varphi(u) > \zeta+1$, $\zeta$  is defined by  (\ref{5.3.2}),
%%Suppose  conditions $(\textbf{\emph{V}})$ and  $(\textbf{\emph{f}}_\textbf{\emph{1}})$-$(\textbf{\emph{f}}_\textbf{\emph{5}})$ hold
%%$\varphi$ is given by (\ref{5.energy1})
%%and $\mathcal{F}$ is a  finite set.
%%Let $M\subset X$ be   bounded, $\tau$-compact and symmetric, i.e. $M \in\tilde{\Sigma} $.  If
%%\begin{equation*}
%%  \bar{\beta} := \sup\limits_{u\in M} \varphi(u) > \zeta+1, \quad \zeta \text{ is defined by  (\ref{5.3.2})},
%%\end{equation*}
then there exists  a  map $h \in \mathcal{H}(M)$ such that $h(M)\in\tilde{\Sigma}$  and
$h(M)\subset \varphi^{\bar{\beta}-1}$.

 \end{lemma}

\begin{proof}
Let $\chi_1(u):\varphi^{\bar{\beta}}_{\beta-3}\rightarrow X$ with $\beta:= \inf\limits_{u\in M} \varphi(u)$ be the vector field given in Lemma \ref{5.vectorfield1},
we study the following Cauchy problem
\begin{equation}\label{donglixitong1}
 \left\{
   \begin{array}{ll}
    \frac{d\eta_1}{dt}=\chi_1(\eta_1)\\
     \eta_1(0,u)=u \in M.
   \end{array}
   \right.
\end{equation}
By the standard theory of ordinary differential equation in Banach space, we know that the initial problem (\ref{donglixitong1}) has a
unique solution on $[0,T_{max})$. Now,
we claim that $T_{max}>1$ for any $u\in M$.
In fact, by Lemma \ref{5.vectorfield1} \textbf{(ii)}, if $T_{max}\leq1$,  for any $t\in [0,T_{max})$ and $u \in M$ we know that
\begin{eqnarray}
  \nonumber\varphi(\eta_1(t,u)) &=& \varphi(u)+\int_{0}^{t}  \frac{d}{ds} \varphi(\eta_1(s,u)) ds\\
     \nonumber&=&   \varphi(u)+\int_{0}^{t}    \langle \nabla\varphi(\eta_1(s,u)), \chi_1(\eta_1(s,u))  \rangle ds \\
     &\geq &   \varphi(u)+\int_{0}^{t} -3  ds  \geq \beta-3T_{max}   \geq\beta-3,\label{varphiMyouxiajie}
\end{eqnarray}
so, $\eta_1(t,u)\subset\varphi^{\bar{\beta}}_{\beta-3}$ for any $(t,u)\in [0,T_{max})\times M$ where $\chi_1(\eta_1(t,u))$ is well defined.
  If $T_{max}\leq 1$, then there exist a sequence $t_m\nearrow T_{max}$ and
$u \in M$ such that  $\|\chi_1(\eta_1(t_{m},u))\|\rightarrow\infty$ as $m\rightarrow\infty$.
By (\ref{5.lessthan mu 8}),  there exists a sequence $\{v_m\}\subset X$ with $\|v_m-\eta_1(t_{m},u)\|_{\tau}\leq \text{min}\{\frac{\mu}{8},\frac{1}{3}\}$
 for all $m\geq 1$ satisfying
 $\nabla\varphi(v_m)\rightarrow 0$ in $X$ and $\theta(\|Q\eta_1(t_{m},u)\|)\neq 0$. Then by Lemma \ref{5.lemma2.6}, passing to a subsequence if necessary, we know that
 either
\\$ \mathbf{(a)} $  there is $z \in [Q \mathcal{F},l_{\bar{\beta}}]$ such that $v_m \in X^{-} \oplus \Big( B_{X^{+}}(z,\frac{\mu}{8})   \Big)$,
\\or,
\\$\mathbf{(b) }$  the sequence $\{v_m\}$ enters infinitely many sets of the form $X^{-} \oplus \Big(   B_{ X^{+} }(z,\frac{\mu}{4}) \Big)$
 where $z \in [Q \mathcal{F},l_{\bar{\beta}}]$.

However, both the cases are impossible. If  $ \textbf{(a)} $  holds, then $Qv_{m}\rightarrow z $ strongly in $X$.
Since $\|v_m\|$ is bounded by Lemma \ref{5.PS bounded},  we also have $Pv_m\rightharpoonup y$ weakly in $X$.
 By the weak continuity of $\varphi'$ we
know  $(y+z) \in \mathcal{K}$ hence $\|v_{m}-\mathcal{K}\|_{\tau}\rightarrow 0$ and this contradicts with
$\theta(\|Q\eta(t_{m},u)\|)\neq 0$.
If  $ \textbf{(b)} $  holds, we may assume that $\eta(t,u)$ leaves $X^{-} \oplus \Big(   B_{X^{+}}(z_1,\frac{\mu}{4})   \Big)$
as $t=t_1$  and  enters $X^{-} \oplus \Big(   B_{ X^{+}}(z_2,\frac{\mu}{4}) \Big)$ as $t=t_2$. Then
$\|\eta_1(t_1,u)-\eta_1(t_2,u)\|\geq \frac{1}{2}\mu$.
by Lemma \ref{5.vectorfield1} \textbf{(ii)},  there exists $\sigma_1>0$ such that

$$\|\chi_1(u)\|\leq \sigma_{1} \text{ for any } u \in \varphi^{\bar{\beta}}\backslash \big( \bigcup\limits_{z \in [Q\mathcal{F},l_{\bar{\beta}}] }  X^{-} 
\oplus   B_{ X^{+}}(z,\frac{\mu}{4})  \big).$$
 So,
\begin{equation*}
    \frac{1}{2}\mu \leq \|\eta_1(t_1,u)-\eta_1(t_2,u)\| \leq\int^{t_2}_{t_1} \|\chi_1(\eta_1(s,u))\| ds\leq  \sigma_{1}(t_2-t_1),
\end{equation*}
which is a contradiction since  $|t_1-t_2|$ can be arbitrarily small as  $t_1$ and $t_2$  arbitrarily close to $T_{max}$.

Next, we claim that    $h(\cdot)= \eta_1(1,\cdot)$.
It is clear that for any $t\in [0,1]$, $\eta_1(t,\cdot)$ is an odd homeomorphism from $M$ to $\eta_1(t,M)$.
 Furthermore, by almost the same way as the proof of \cite[Lemma 6.8]{MW} we can get that $\eta_1$ is $\tau$-continuous and
for any $(t,u)\in [0,1] \times M$ there exist a neighborhood $W_{(t,u)}$ of $(t,u)$ in
          the $|\cdot|\times\tau$-topology such that
          $$\{  v-\eta_1(s,v)| (s,v)\in  W_{(t,u)}\cap  ([0,1] \times M)   \}$$
          is contained in a finite-dimensional subspace of $X$.
%\varphi^{\bar{\beta}} _{\beta-3 }-->M

Then, we claim that  $\eta_1(1, M) \subset \varphi^{  \bar{\beta} -1} $,  $\bar{\beta}$ is given by Lemma \ref{5.first deformation lemma}.
 Suppose there exists $u\in M$
such that $\varphi(\eta_1(1, u)) > \bar{\beta}-1$, then $\varphi(\eta_1(t, u))>  \bar{\beta}-1$ for any $t\in [0,1]$,
since (by Lemma \ref{5.vectorfield1} \textbf{(ii)})
$\varphi(\eta_1(t, u))$ is nonincreasing  in $t$. So, $\eta_1(t, u) \in E_1$ with $E_1$ given in Lemma \ref{5.vectorfield1} \textbf{(iii)}, for any  $(t, u) \in [0,1]\times M$, then
\begin{eqnarray*}
  \varphi(\eta_1(1,u) &=& \varphi(\eta_1(0,u)+ \int_{0}^{1}  \langle \nabla\varphi(\eta_1(s,u)), \chi_1(\eta_1(s,u)) \rangle ds\\
         &<& \varphi(\eta_1(0,u)+ \int_{0}^{1}  -1 ds \leq \bar{\beta}-1,
\end{eqnarray*}
which lead to contradiction.

Finally, in order to prove $\eta_1(1,\cdot) \in \mathcal{H}(M)$,  we only need to show that $\eta_1(1,M)$  is  $\tau$-compact and bounded.
Indeed, the $\tau$-compactness of  $\eta_1(1,M)$  follows directly since $\eta_1(1,\cdot)$ is $\tau$-continuous.
Thus, $Q(\eta_1(1,M))$ is bounded in $X^+$.
On the other hand, by (\ref{varphiMyouxiajie}) we know   $\varphi(\eta_1(1,u))\geq\beta-3$ for any $u \in M$.
Together with (\ref{xiajie to youjie}) we have $\eta_1(1,M)$  is   bounded.
\iffalse On the other hand,
$\eta_1$ is an admissible homotopy, then $\{\eta(t,u):(t,u) \in  [0,1] \times M\}$ is also $\tau$-compact.
Noting that $\varphi(\eta(t,u))>\beta-3$ for any $(t,u) \in  [0,1] \times M$,
by a similar argument as that following
(\ref{5.3.3}), we have that there exists $\bar{C}>0$ such that
\begin{equation*}
    \|\chi(\eta(t,u))\|\leq \bar{C} \quad \text{for any} \quad (t,u) \in [0,1]\times M,
\end{equation*}
then the boundedness of $\eta(1,M)$ can be proven by
\begin{eqnarray*}
  \|\eta(1,u) \|&\leq& \|u\|+ \int_{0}^{1} \|\chi(\eta(s,u)) \| ds \\
     &\leq& \|u\|+ \int_{0}^{1} \bar{C} ds \\
     &\leq& \|u\|+ \bar{C}.
\end{eqnarray*}
\fi
So, let $h(\cdot)=  \eta_1(1,\cdot) $,   the proof is complete.
\end{proof}

Next we define a new class of maps related to $\mathcal{H}(A)$ for convenience.
\begin{definition}\label{5.def of H tild A}
For $A \in \tilde{\Sigma}$ ($\tilde{\Sigma}$ is defined by (\ref{defofsigematil})),  define   $\tilde{\mathcal{H}}(A)$
 to be the set of all maps $h: A\rightarrow X$ satisfying:
\begin{itemize}
  \item $h$ is a homeomorphism of $A$ onto $h(A)$ in the original topology, i.e., the topology induced by $\|\cdot\|$, of $X$;
  \item $h$ is an odd admissible map which maps bounded set into bounded set;
  \item  $\varphi(h(u)) \leq \varphi(u)$ for any $u \in A\cap \varphi_{ a -1}$ and $\varphi(h(u)) \leq  a -1$
for any $u \in A\cap \varphi^{ a -1}$, where $ a $ is given by Lemma \rm{\ref{5.lemma2.1}\textbf{(ii)}}.
\end{itemize}
\end{definition}
 Clearly  $ \mathcal{H}(A) \subset  \tilde{\mathcal{H}}(A)$ for any $A \in \tilde{\Sigma}$, so, for $h$
 given in Lemma \ref{5.first deformation lemma} we have $h\in\tilde{\mathcal{H}}(M)$.
%%and $\tilde{\mathcal{H}}(A)$ is also closed under composition.

\begin{lemma}\label{5.vectorfield2}
 Under conditions $(\textbf{\emph{V}})$, $(\textbf{\emph{f}}_\textbf{\emph{1}})$-$(\textbf{\emph{f}}_\textbf{\emph{5}})$, let
%%$\varphi$ is given by (\ref{5.energy1})
 $\mathcal{F}$ be a  finite set and let $M \in \Sigma  $ be bounded in $X$,
then for any  $R> \rho:= \sup\{\|u\|: u \in M \}$,
there exists  a vector field $\chi_2(u):\varphi^{\zeta+2}\rightarrow X$ ($\zeta$  is defined by  (\ref{5.3.2})) with the following properties:
\begin{description}
  \item[(i)] $\chi_2(u)$ is locally Lipschitz continuous   and $\tau$-locally Lipschitz $\tau$-continuous on $\varphi^{\zeta+2}$.
  \item[(ii)] $\chi_2(u)$ is odd and   $ \langle\nabla\varphi(u),\chi_2(u)\rangle  \leq 0$,  for any $u \in \varphi^{\zeta+2}_{ a -1}$,
                 where $ a $ is given by Lemma \rm{\ref{5.lemma2.1}\textbf{(ii)}}.
  \item[(iii)] $ \langle\nabla\varphi(u),\chi_2(u)\rangle  <-1$,  for any
               $$u\in \varphi_{a-1}^{\zeta+2} \cap\{u\in X:\|u\|_{\tau}\geq\delta_{0},\|Qu-[Q\mathcal{F},l_{\zeta+2}]\backslash \{0\}\|\geq
                 \frac{\mu}{8}\}  \cap \bar{B}_{R},$$
                where $\bar{B}_{R}=\{u\in X:\|u\|\leq R\}$, $\delta_{0}$ is given by (\ref{5.r}), $l_{\zeta+2}$ and $\mu$ are given in
                 Lemma \ref{5.lemma 2.7}.
\item[(iv)] There exists $\sigma_2>0$ such that $\|\chi_2(u)\|<\sigma_2$  for  any    $ u\in  \varphi^{\zeta+2}$.
\item[(v)] Each $u\in \varphi^{\zeta+2}$ has  a $\tau$-open neighborhood of $u$, $U_{u}\subset X$,
            such that   $\chi_2(U_{u}\cap\varphi^{\zeta+2})$ is contained in a finite-dimensional subspace of $X$.
%%\item[(v)] If $A$ is $\tau$-compact and $\varphi(A)$ is bounded from below, then
%%              \begin{equation} \label{5.3.3}
%%                 \sup\limits_{u\in A}\|\chi(u)\|<\infty.
%%             \end{equation}
\end{description}
\end{lemma}

\begin{proof}
%%Before construct the map $h$, some notations are needed.
Let
\begin{equation*}
     N_{0}=X^- \oplus \bigcup\limits_{z \in[Q\mathcal{F},l_{\zeta+2}]\backslash \{0\}} \big(  B_{X^+} (z,\frac{\mu}{16})  \big),
\end{equation*}
and
\begin{equation*}
    E_2=\varphi^{\zeta+2} \cap\{u\in X:\|u\|_{\tau}\geq\frac{\delta_{0}}{2}\}  \cap \{X\backslash  N_{0}\},
\end{equation*}
where $l_{\zeta+2}$ and $\mu$ are given by Lemma \ref{5.lemma 2.7}, clearly,
 there exists   $\sigma>0$    such that
\begin{equation}\label{5.lemma3.3 1}
   \| \nabla\varphi(u)\| >\sigma   \text{ for any }   u \in E_2.
%\varphi^{\zeta+2}\backslash E --->E
\end{equation}
Define
 $$\omega(u)=\frac{2\nabla\varphi(u)}{\|\nabla\varphi(u)\|^{2}}, \text { for } u\in E_2\cap \bar{B}_{R}.$$
Similar to Lemma \ref{5.vectorfield1}, by the weak continuity of $\nabla\varphi$  and (\ref{xiajie to youjie}), for any $u\in E_2\cap \bar{B}_{R}$,
 there exists a $\tau$-neighborhood of $u$, $V_{u}\subset X$, such that
\begin{equation}\label{5.lemma3.3 2}
     \langle \nabla\varphi(v),  \omega(u) \rangle >1 , \text{ for any } v\in V_{u}\cap \varphi_{a-1} \cap E_2\cap \bar{B}_{R}.
\end{equation}
Clearly, (by the convexity of $\bar{B}_{R}$)  $\bar{B}_{R}$ is $\tau$-closed, so, $X \backslash \bar{B}_{R}$ is $\tau$-open, then

$$\mathcal{N}_2=\{V_{u}:u\in E_2\cap \bar{B}_{R}\} \cup \Big(X\backslash \bar{B}_{R} \Big)$$
forms a $\tau$-open covering of $E_2$.
Since $\mathcal{N}_2$ is metric and paracompact, that is, there exists a locally finite $\tau$-open covering $\mathcal{M}_2=\{M_{i}:i\in \Lambda \}$ of $E_2$ finer than $\mathcal{N}_2$.
Using a similar partition of unity argument as Lemma 3.1, we can construct a pseudogradient vector field on $\varphi^{\zeta+2}$.
To be specific,
 take $\omega_{i}=\omega(u_{i})$ if $M_{i}\subset V_{u_{i}}$ for some $u_{i}\in E$
and take $\omega_{i}=0$ if $M_{i}\subset X\backslash B_{R}$.
 Let

$$\xi_2(u)=\sum_{i \in \Lambda}\lambda_{i}(u)\omega_{i}, ~u\in \mathcal{N}_2,$$
where $\{\lambda_{i}:i \in \Lambda\}$ be a $\tau$-Lipschitz continuous partition
of unity subordinated to $\{M_i\}$.
Since the $\tau$-open covering $\{M_i\}$ is locally finite, each $u\in\mathcal{N}_2$ belongs
to finite many sets $M_{i}$. Therefore, for every $u\in \mathcal{N}_2$, the sum $\xi_2(u)$ is only a finite sum.
It follows that, for any $u\in \mathcal{N}_2$, there exists a $\tau$-open neighborhood of $u$, $U_{u} \subset \mathcal{N}_2$,
such that  $\xi_2(U_{u})$ is contained in a finite-dimensional subspace of $X$.
By the fact that all norms for a finite-dimensional vector space are equivalent, we know that there exists $L_{u}>0$ such that

$$\|\xi_2(v)-\xi_2(w)\|\leq L_{u} \|v-w\|_{\tau},   \text{ for any }   v,w \in U_{u}.$$
Then it is easy to see that $\xi_2(u)$ is locally
Lipschitz continuous and $\tau$-locally Lipschitz $\tau$-continuous.
By (\ref{5.lemma3.3 1}) and (\ref{5.lemma3.3 2}), we also have

$$\langle \nabla\varphi(u),\xi_2(u)\rangle >1 \text{ and }  \|\xi_2(u)\|<\frac{2}{ \sigma}:=\sigma_2,
\text{ for any } u\in E_2\cap \bar{B}_{R}\cap \varphi_{a-1},$$
and

$$\langle \nabla\varphi(u),\xi_2(u)\rangle\geq 0 ,
\text{ for any  } u\in E_2\cap \varphi_{a-1}.$$
Define

$$ \tilde{\xi}_2(u)=\frac{\xi_2(u)-\xi_2(-u)}{2},$$
and take the following  two Lipschitz continuous and $\tau$-Lipschitz $\tau$-continuous cut-off   functions:

 $$  \vartheta(u)=
 \left\{
   \begin{array}{ll}
      1,   \text{ if }  \|u\|_{\tau}\geq \delta_0,\\
      0,   \text{ if }  \|u\|_{\tau}\leq \frac{2\delta_0}{3},
   \end{array}
   \right.$$
and

 $$  \psi(u)=
 \left\{
   \begin{array}{ll}
     1,  \text{ if }  \|Qu-z\|\geq \frac{\mu}{8}, \text{ for any } z \in[Q\mathcal{F},l_{\zeta+2}]\backslash \{0\}, \\
     0,   \text{ if }  \|Qu-z\|\leq \frac{\mu}{10}, \text{ for any } z \in[Q\mathcal{F},l_{\zeta+2}]\backslash \{0\}.
   \end{array}
   \right.$$
Define the vector field $\chi_2: \varphi^{\zeta+2}\rightarrow X$ by

$$  \chi_2(u)=
 \left\{
   \begin{array}{ll}
     -\vartheta( u )\psi(u)\tilde{\xi}_2(u), \text{ for }  u\in \mathcal{N}_2,\\
     0,   \text{ for }  \|Qu\| \leq\frac{2}{3}\delta_0  \text{ or } \|Qu-[Q\mathcal{F},l_{\zeta+2}]\backslash \{0\}\|\leq \frac{\mu}{10}.
   \end{array}
   \right.$$
Then, it is easy to see that $\chi_2$ is well defined on  $\varphi^{\zeta+2} $
and satisfies the   properties \textbf{(i)}-\textbf{(v)}.
\end{proof}

Before giving our second deformation lemma, we introduce a new class of admissible   maps related to $\tilde{\mathcal{H}}(A)$.

\begin{definition}\label{5.def of H tild H}
Let  $H\subset X$ be a fixed set,    define
  $\tilde{\mathcal{H}}_H$ is the set   of all maps $h: H \rightarrow X$ satisfying
\begin{itemize}
  \item $h$ is an admissible odd map and  $h$ is  homeomorphism in the original topology of $X$;
  \item  if $A\subset H$  is  bounded, then $h(A)$ is also  bounded;
  \item $\varphi(h(u)) \leq \varphi(u)$ for any $u \in H\cap \varphi_{a-1}$ and $\varphi(h(u)) \leq a-1$
for any $u \in H\cap \varphi^{a-1}$, where  $a$   is given in Lemma \rm{\ref{5.lemma2.1}\textbf{(ii)}}.
\end{itemize}
\end{definition}
Clearly, for any $h \in \tilde{\mathcal{H}}_H$ and $A \in \tilde{\Sigma}$ with  $A \subset H$, we have $h\mid_{A}\in \tilde{\mathcal{H}}(A)$.

For $A \in  \Sigma$  we
recall the Krasnoselskii genus $\gamma(A)$ of $A$ (see section 7 of \cite{Minimax-methods}):

\begin{equation}\label{5.def of genus}
    \gamma(A):= \min \{ k \in \mathbf{N}:  \exists \text{ odd continuous map } \phi: A\rightarrow \mathbb{R}^{k}\backslash \{0\} \};
~ \gamma(\emptyset) :=0.
\end{equation}
Then we have the  following second deformation lemma.

\begin{lemma}\label{5.second deformation lemma}{\bf(Deformation Lemma 2)}
Under conditions $(\textbf{\emph{V}})$, $(\textbf{\emph{f}}_\textbf{\emph{1}})$-$(\textbf{\emph{f}}_\textbf{\emph{5}})$, let
 $\mathcal{F}$ be a  finite set and let $M \in\tilde{\Sigma} $.  If
\begin{equation*}
 b\leq \bar{\beta} := \sup\limits_{u\in M} \varphi(u) < \zeta+2, \text{ with } b \text{   given in  } (\ref{5.r}),
\end{equation*}
and
\begin{equation}\label{deformationlemma2epsilon}
   0<\epsilon< \min \{ b-\sup\limits_{\|u\|_{\tau}\leq \delta_0} \varphi(u), \frac{\mu}{8\sigma_2}, 1 \} \text{ with } \sigma_2 \text{ given in Lemma \rm{\ref{5.vectorfield2} \textbf{(iv)}}},
\end{equation}
%where  $\sigma_2$  is given in Lemma \ref{5.vectorfield2} \textbf{(iv)},
then there exist % $\epsilon\in(0,1)$,
a symmetric $\tau$-open set $\mathcal{N}$ with $ \bar{\mathcal{N}} $   $\tau$-closed and $\gamma(\bar{\mathcal{N}})=1$
and  $h\in \tilde{\mathcal{H}}_{\varphi^{\zeta+2}}$, such that
$h(M\backslash\mathcal{N} )\in\tilde{\Sigma}$   and

\begin{equation*}
    h(M\backslash\mathcal{N})\subset \varphi^{\bar{\beta}-\epsilon}.
\end{equation*}

\end{lemma}

\begin{proof}
Define

\begin{equation*}
    \mathcal{N}:=\bigcup\limits_{z \in[Q\mathcal{F},l_{\zeta+2}]\backslash \{0\}} \big( X^- \oplus  B_{X^+} (z,\frac{\mu}{4})  \big)
     =   X^- \oplus \bigcup\limits_{z \in[Q\mathcal{F},l_{\zeta+2}]\backslash \{0\}} \big(  B_{X^+} (z,\frac{\mu}{4})  \big),
\end{equation*}
where $l_{\zeta+2}$ and $\mu$ are given in  Lemma \ref{5.lemma 2.7}.
Note $ [Q\mathcal{F},l_{\zeta+2}] $ is countable and for any $z \in[Q\mathcal{F},l_{\zeta+2}]$
 $\big( X^- \oplus  B_{X^+} (z,\frac{\mu}{4})  \big)$ is contractible, then

\begin{equation*}
    \gamma(\bar{\mathcal{N}})=1.
\end{equation*}
%Fix an $\epsilon$ satisfying
%\begin{equation}\label{deformationlemma2epsilon}
%   0<\epsilon< \min \{ b-\sup\limits_{\|u\|_{\tau}\leq \delta_0} \varphi(u), \frac{\mu}{8\sigma_2}, 1 \},
%\end{equation}
%where  $\sigma_2$  is given in Lemma \ref{5.vectorfield2} \textbf{(iv)}, and let
%Let
%\begin{equation}\label{deformationlemma2R}
 %        R=  \sigma_2 \epsilon+\rho+1, ~\rho:= \sup\{\|u\|: u \in M \},
%\end{equation}

Let $\chi_2(u):\varphi^{\zeta+2}\rightarrow X$ be the  vector field given in Lemma \ref{5.vectorfield2},
  we  consider the following Cauchy problem

$$
 \left\{
   \begin{array}{ll}
    \frac{d\eta_2}{dt}=\chi_2(\eta_2)\\
     \eta_2(0,u)=u \in \varphi^{\zeta+2} .
   \end{array}
   \right.$$
By the standard theory of ordinary differential equation in Banach space, we know that the initial problem has a
unique solution on $[0,\infty)$. Clearly, $\eta_2(t,u)$ is odd in $u$, furthermore, the similar argument as the proof of Lemma 6.8 of \cite{MW} yields
that $\eta_2$ is an admissible homotopy.   By Lemma \ref{5.vectorfield2} \textbf{(ii)},  we   have

$$\frac{d}{dt}\varphi(\eta_2(t,u))=\langle \nabla\varphi(\eta_2(t,u)) , \chi_{2}(\eta_2(t,u)) \rangle\leq 0, \text{ for }  \eta_2(t,u) \in \varphi^{\zeta+2}_{a-1},$$
where  $a$   is given in Lemma \rm{\ref{5.lemma2.1}\textbf{(ii)}}.

For any  $\epsilon$ given by (\ref{deformationlemma2epsilon}) and %$R$ given by (\ref{deformationlemma2R}),
let

\begin{equation*}\label{deformationlemma2R}
         R=  \sigma_2 \epsilon+\rho+1, ~\rho:= \sup\{\|u\|: u \in M \},
\end{equation*}
 we claim that $\{\eta_2(t,u): 0\leq t \leq \epsilon , u\in M\}\subset B_{R}$.
Indeed, since

$$\eta_2(t,u)=u+\int_{0}^{t} \chi_{2}(\eta_2(s,u))  ds,$$
by Lemma \ref{5.vectorfield2} \textbf{(iv)} we know that, for any $t\in[0,\epsilon]$ and $u \in M$,

\begin{eqnarray}
  \nonumber\|\eta_2(t,u) \|&\leq& \|u\|+ \int_{0}^{t} \|\chi_2(\eta_2(s,u)) \| ds \\
     &\leq& \|u\|+ \int_{0}^{t} \sigma_2 ds  \leq  \rho+   \sigma_2 \epsilon  <R. \label{deformationlemma2bound}
\end{eqnarray}
Then by (\ref{deformationlemma2epsilon}), we have

$$\sup\limits_{\|u\|_{\tau}\leq \delta_0}\varphi(u)<b-\epsilon,$$
that is,

$$\{u\in X: \|u\|_{\tau}\leq \delta_0 \}\subset \varphi^{b-\epsilon}.$$
 Moreover, by Lemma \ref{5.vectorfield2} \textbf{(iv)}, for $t\in[0,\infty)$,

\begin{equation*}
  \|\eta_2(t,u)-u \| \leq   \int_{0}^{t} \|\chi_2(\eta_2(s,u)) \| ds
      \leq  \int_{0}^{t} \sigma_2 ds
       \leq    \sigma_2 t.
\end{equation*}

Let $t=\epsilon$, this gives

\begin{equation*}
    \|\eta_2(\epsilon,u)-u \|\leq  \sigma_2 \epsilon<\frac{\mu}{8},
\end{equation*}
here    $\epsilon< \frac{\mu}{8\sigma_2}$ is required. Then, for any $u\in M\backslash\mathcal{N}$, we have

\begin{equation*}
   \eta_2(\epsilon,u) \in  \varphi^{\zeta+2} \cap  \Big( X^- \oplus \bigcup\limits_{z \in[Q\mathcal{F},l_{\zeta+2}]\backslash \{0\}} \big(  B_{X^+} (z,\frac{\mu}{8})  \big) \Big).
\end{equation*}

Next, we claim that $\eta_2(\epsilon, M\backslash \mathcal{N})\subset \varphi^{\bar{\beta}-\epsilon}$.
If  there exists $u \in M\backslash \mathcal{N}$ such that $\eta_2(\epsilon, u)\geq\bar{\beta}-\epsilon$,
then, by Lemma \ref{5.vectorfield2} \textbf{(iii)} we have

\begin{eqnarray*}
  \varphi(\eta_2(\epsilon,u) &=& \varphi(\eta_2(0,u)+ \int_{0}^{\epsilon}  \langle \nabla\varphi(\eta_2(s,u)), \chi_2(\eta_2(s,u)) \rangle ds\\
         &<& \varphi(u)+ \int_{0}^{\epsilon}  -1 ds
         \leq \bar{\beta}-\epsilon,
\end{eqnarray*}
which is a contradiction.

Finally,
by Lemma \ref{5.vectorfield2} \textbf{(ii)}, it is easy to see
$\varphi(\eta_2(\epsilon,u)) \leq    \varphi(u)$ if $ u \in  \varphi^{\zeta+2}_{a-1}$ and
$  \eta_2(\epsilon, \varphi^{a-1})  \subset \varphi^{a-1}$.
The boundedness of $M$ follows directly from (\ref{deformationlemma2bound}).
Further more,  if $M$ is $\tau$-compact, then  $M\backslash\mathcal{N}$, $M\cap \bar{\mathcal{N}}$  and  $\eta_2(\epsilon,  M\backslash\mathcal{N}) $ are also $\tau$-compact since $\eta_2(\epsilon,\cdot)$ is $\tau$-continuous. Let $h(\cdot)=  \eta_2(\epsilon,\cdot) $,   the proof is complete.
\end{proof}

\section{Proof of Theorem \ref{5.main result}}
~

  This section is devoted to prove our Theorem \ref{5.main result}. For this purpose,    we introduce two  kinds of
 pseudo-$Z_{2}$ indexes.
%%The difference is the   indexes are only defined on $\tilde{\Sigma}$ and the map class $\tilde{\mathcal{H}}$ given below is
%%more general than $\mathcal{H}$.

For $A \in \tilde{\Sigma}$ (defined by (\ref{defofsigematil})), we define    $\gamma^*(A)$ as

\begin{equation*}
    \gamma^*(A):=\min\limits_{h \in \tilde{\mathcal{H}}(A)} \gamma(h(A)\cap S_r\cap X^+),\quad A \in \tilde{\Sigma},
\end{equation*}
where $\gamma$ is the genus (see (\ref{5.def of genus})), $\tilde{\mathcal{H}}(A)$ is defined by Definition \ref{5.def of H tild A}
 and $r$ is obtained in (\ref{5.r}). The following lemma gives some  properties of $\gamma^*$, its
 proof can be found in \cite{SK}, so we omit here.
\begin{lemma}\label{5.prop of genus}
Let $A,B \in \tilde{\Sigma}$.
\begin{description}
  \item[(i)] If $\gamma^*(A)\neq 0$, then $A\neq \emptyset$.
  \item[(ii)] If $A\subset B$, then $\gamma^*(A)\leq\gamma^*(B)$.
  \item[(iii)] If $h \in \tilde{\mathcal{H}}(A)$, then  $\gamma^*(h(A))\geq\gamma^*(A)$.
\end{description}
\end{lemma}

\begin{remark}
There are sets of arbitrarily large pseudoindex, $\gamma^*$, in $\bar{\Sigma}$. For any $k\in\mathbb{Z}^+$,
%%there exists  $A \in \tilde{\Sigma}$ such that $\gamma^*(A)\geq k$. Indeed,
let $X_{k}:=\overline{(\bigoplus_{j=1}^{k}\mathbb{R}f_{j})\oplus X^-}$, where $\{f_{j} \}_{j\geq 1}$ is
an  orthonormal  basis of  $X^+$.
Then, by Lemma \ref{5.lemma2.1}\textbf{(ii)},
 there exists $R_k>r$ such that

$$\sup\limits_{u\in X_k,\|u\|=R_k} \varphi(u)   <\inf\limits_{\|u\|\leq r} \varphi(u).$$
Take

$$A:= \{u\in X_k : \|u\|\leq R_k \}.$$
Clearly, $A$ is bounded and $\tau$-compact.
Then, noting Lemma \rm{\ref{5.lemma2.1}\textbf{(ii)}} and arguing exactly as Lemma 4.5 of \cite{Bartsch-Ding} (see also Proposition 7 of \cite{Rup}), we know that

 $$\gamma^*(A)\geq k.$$
\end{remark}
Furthermore, we  need another  pseudo-$Z_{2}$ index defined on $ \tilde{ \Sigma}_H$ with

\begin{equation*}
     \tilde{ \Sigma}_H:= \{ A \in \tilde{ \Sigma} : A \subset H\},
\end{equation*}
where   $H\subset X$ is a fixed set.
For any $ A \in \tilde{ \Sigma}_H$, we define

\begin{equation*}\label{5.def of pseudoindex2}
    \gamma_H^*(A):=\min\limits_{h \in \tilde{\mathcal{H}}_H} \gamma(h(A)\cap S_r\cap X^+),
\end{equation*}
where $\tilde{\mathcal{H}}_H$ is defined by Definition \ref{5.def of H tild H}.
Here are some basic properties of $\gamma_{H}^*$:

\begin{lemma}\label{5.prop of new genus}
Let  $A, B \in \tilde{\Sigma}_H$, then
\begin{description}
  \item[(i)]  If $\gamma_H^*(A)\neq 0$ then $A\neq \emptyset$.
  \item[(ii)] $\gamma_H^*(A) \geq \gamma^*(A)$.
  \item[(iii)] If $A\subset B$ then $\gamma_H^*(A)\leq\gamma_H^*(B)$.
  \item[(iv)] If $h \in \tilde{\mathcal{H}}_H$ and  $h(A)\subset H $ then  $\gamma_H^*(h(A))\geq\gamma_H^*(A)$.
  \item[(v)] $\gamma_H^*(A\cup B)\leq\gamma_H^*(A) + \gamma(B)$.
\end{description}
\end{lemma}

\begin{proof}
$\mathbf{(i)}$, $\mathbf{(iii)}$ and $\mathbf{(iv)}$ follow directly from the properties of genus,
here we only give a simple proof of $\mathbf{(ii)}$  and $\mathbf{(v)}$.
Firstly, if  $h \in \tilde{\mathcal{H}}_H$, then $h|_{A} \in \tilde{ \mathcal{H}} (A)$, so $\gamma_H^*(A) \geq \gamma^*(A)$,
thus  $\mathbf{(ii)}$ is proved.
Then for any $h \in \tilde{\mathcal{H}}_H$, by the  subadditivity of genus  we have

\begin{equation*}
    \gamma_H^*(A\cup B)\leq  \gamma(h(A\cup B)\cap S_r\cap X^+)\leq    \gamma(h(A)\cap S_r\cap X^+)+\gamma(h|_{B}(B)).
\end{equation*}
Since $h|_{B}: B\rightarrow h|_{B}(B) $ is a homeomorphism,  we know that $\gamma(h|_{B}(B))=\gamma(B)$, and

\begin{equation*}
   \gamma_H^*(A\cup B)\leq  \gamma(h(A)\cap S_r\cap X^+)+\gamma(B),
\end{equation*}
the proof of $\mathbf{(v)}$ is complete.
\end{proof}

Before giving the proof of Theorem \ref{5.main result}, we define a mini-max value through $\gamma^*$:

$$c_k:= \inf\limits_{\gamma^*(A)\geq k} \sup\limits_{u \in A} \varphi(u),~ A \in \tilde{\Sigma},$$
then $c_k$  is  well-defined  for each $k\geq 1$. Moreover,

\begin{equation*}
    b\leq c_k\leq c_{k+1}, \text{ for any integer }  k\geq 1,
\end{equation*}
where $b$ is defined by (\ref{5.r}).
Now we are ready to the prove Theorem \ref{5.main result}.

\begin{proof}[\textbf{Proof of Theorem \ref{5.main result}}]
By contradiction, if $\mathcal{F}$ (defined by (\ref{defofhuaF})) is finite, i.e., (\ref{5.equation}) has
only finite many geometrically distinct solutions,   there  are two possibilities for $c_k$:

\begin{description}
  \item[Case (I)] There is an integer $k\geq 1$ such that $c_k > \zeta+1$ ($\zeta$ is given in (\ref{5.3.2})), or,
  \item[Case (II)] $b\leq c_k \leq \zeta+1$ for all $k\geq1$.
\end{description}
In what follows, we show that both cases \textbf{(I)} and \textbf{(II)} are impossible by using the first and second deformation lemmas.
If case \textbf{(I)} holds, that is,  there exists an integer $k\geq 1$ such that

\begin{equation*}\label{5.ck fanwei}
    c_k > \zeta+1,
\end{equation*}
then, there exists  $M\in \tilde{\Sigma}$ with $\gamma^*(M)\geq k$
and

\begin{equation*}
  \zeta+1< \sup\limits_{u\in M} \varphi(u)<c_k+\frac{1}{2}.
\end{equation*}
Hence, by the first deformation  Lemma \ref{5.first deformation lemma}, there exists $h \in \mathcal{H}(M)$, thus $h \in \tilde{\mathcal{H}}(M)$, such that

\begin{equation*}
   \sup\limits_{u\in M} \varphi(h(u))<c_k-\frac{1}{2}.
\end{equation*}
However, by  Lemma \ref{5.prop of genus}\textbf{(iii)} we have $\gamma^*(h(M)) \geq  \gamma^*(M) \geq k$, so,

 $$c_k\leq  \sup\limits_{u \in h(M)} \varphi(u)<c_k-\frac{1}{2},$$
which  is a contradiction, so, case \textbf{(I)} is false.

If  case \textbf{(II)} occurs, then
\begin{equation}\label{5.ck shangquejie}
  b \leq  \bar{c}:= \lim\limits_{k\rightarrow\infty} c_k\leq \zeta+1.
\end{equation}
%%In order to rule out case $\mathbf{(jj)}$, we need another  mini-max value.
Fix $H= \varphi^{\zeta+2}$ with $\zeta$ given by \ref{5.3.2} and define

\begin{equation*}
     \tilde{c}_k:= \inf\limits_{\gamma_H^*(A)\geq k} \sup\limits_{u \in A} \varphi(u), \text{ for } A \in \tilde{\Sigma}_H.
\end{equation*}
By Lemma \ref{5.prop of new genus} $\mathbf{(ii)}$  we have

\begin{equation*}
    \tilde{c}_k  \leq   c_k,
\end{equation*}
so, $\tilde{c}_k $ is also bounded from above, then for any integer    $ k\geq 1$ we have

\begin{equation*}
    b \leq \tilde{c}_k \leq \tilde{c}_{k+1} \leq  \tilde{c}:= \lim\limits_{k\rightarrow\infty} \tilde{c}_k\leq  \bar{c} \leq \zeta+1,
\end{equation*}
where $\bar{c}$ is defined by (\ref{5.ck shangquejie}).

Let $\epsilon \in (0,1)$ be   given in Lemma \ref{5.second deformation lemma}, then there exist $k_1\in \mathbf{N}$ satisfying

\begin{equation}\label{5.c tild used to contradiction}
    \tilde{c}-\frac{\epsilon}{4}  < \tilde{c}_{k_1}\leq \tilde{c},
\end{equation}
and   a set  $M_1 \in \tilde{\Sigma}_H$  with $\gamma_H^*(M_1)\geq k_1$  such that

\begin{equation*}
     \tilde{c}-\frac{\epsilon}{4}  < \tilde{c}_{k_1}\leq \sup\limits_{u\in M_1} \varphi(u)<\tilde{c}+\frac{\epsilon}{4}<\zeta+2.
\end{equation*}
On the other hand, for any $M \in \tilde{\Sigma}_H$, let $\mathcal{N}$ be  a  $\tau$-open set
 with $\gamma(\mathcal{\bar{N}})=1$, by Lemma \ref{5.prop of new genus}\textbf{(iii)}\textbf{(v)},  we have

\begin{eqnarray*}
  \gamma_H^*(M) &=& \gamma_H^*\big(   (M\backslash \mathcal{N}) \cup (M\cap \bar{\mathcal{N}})   \big) \leq
\gamma_H^*(M\backslash \mathcal{N})  + \gamma(M\cap \bar{\mathcal{N}}) \\
   &\leq&   \gamma_H^*(M\backslash \mathcal{N})  + \gamma(\bar{\mathcal{N}})
   =    \gamma_H^*(M\backslash \mathcal{N})  +1,
\end{eqnarray*}
i.e.,

\begin{equation}\label{5.genus estimate}
   \gamma_H^*(M\backslash \mathcal{N})\geq \gamma_H^*(M)-1\text{.}
\end{equation}
Now, we are going to deduce a contradiction whenever  $\gamma_H^*(M_1)$ is infinite or finite.

If $\gamma_H^*(M_1)=+\infty$, then $\gamma_H^*(M_1\backslash \mathcal{N})=+\infty$. By
Definition \ref{5.def of H tild H} with $H= \varphi^{\zeta+2}$ and     the second deformation Lemma \ref{5.second deformation lemma},
there exists  $h \in \tilde{\mathcal{H}}_H$ such that

\begin{equation*}
   \sup\limits_{u\in M_1\backslash \mathcal{N}} \varphi(h(u))< \tilde{c}+\frac{\epsilon}{4}-\epsilon=\tilde{c}-\frac{3\epsilon}{4}.
\end{equation*}
Lemma \ref{5.prop of new genus}\textbf{(iv)} implies that  $\gamma_H^*(h(M_1\backslash \mathcal{N}))\geq \gamma_H^*(M_1\backslash \mathcal{N})=+\infty$, hence

\begin{equation*}
    \tilde{c}_{k_1} \leq  \sup\limits_{u \in h( M_1\backslash \mathcal{N})} \varphi(u)< \tilde{c} -\frac{3\epsilon}{4},
\end{equation*}
which   contradicts with (\ref{5.c tild used to contradiction}).

If $\gamma_H^*(M_1)= \Gamma <+\infty$ with $\Gamma\geq k_1$, we have

\begin{equation*}
     \tilde{c}-\frac{\epsilon}{4}  < \tilde{c}_{k_1}\leq \tilde{c}_{\Gamma+1}\leq \tilde{c},
\end{equation*}
then there exists an $M_2 \in \tilde{\Sigma}_H$  with   $\gamma_H^*(M_2)\geq \Gamma+1$    such that

\begin{equation*}
     \tilde{c}-\frac{\epsilon}{4}  < \tilde{c}_{\Gamma+1}\leq \sup\limits_{u\in M_2} \varphi(u)<\tilde{c}+\frac{\epsilon}{4}.
\end{equation*}

By Lemma \ref{5.second deformation lemma},
there exists $h \in \tilde{\mathcal{H}}_H$ such that

\begin{equation*}
   \sup\limits_{u\in M_2\backslash \mathcal{N}} \varphi(h(u))<\tilde{c}+\frac{\epsilon}{4}-\epsilon=\tilde{c}-\frac{3\epsilon}{4}.
\end{equation*}
On the other hand, by (\ref{5.genus estimate}) we have

\begin{equation*}
 \gamma_H^*(h(M_2\backslash \mathcal{N}))\geq \gamma_H^*(M_2\backslash \mathcal{N})\geq \gamma_H^*(M_2)-1\geq \Gamma \geq k_1.
\end{equation*}
So,

\begin{equation*}
    \tilde{c}_{k_1} \leq  \sup\limits_{u \in h( M_2\backslash \mathcal{N})} \varphi(u)<\tilde{c}-\frac{3\epsilon}{4},
\end{equation*}
which also   contradicts with (\ref{5.c tild used to contradiction}).  The proof is complete.
\end{proof}

%For acknowledgements section, please don't number the section, please begin it with \section*{Acknowledgements}
%\section*{Acknowledgments} We would like to thank you for \textbf{following
%the instructions above} very closely in advance. It will definitely
%save us lot of time and expedite the process of your paper's publication.

% You may incorporate your references as follows in your main tex file.
% Using BibTex is not recommended but can be handled.

 {}


\vskip6mm


 \begin{thebibliography}{}



\bibitem{Ackermann}  N. Ackermann, On a periodic Schr$\ddot{\text{o}}$dinger equation with nonlocal superlinear part, Math. Z. 248
           (2004), 423-443.



\bibitem{AlamaLi1} S. Alama, Y. Y. Li, Existence of solutions for semilinear elliptic equations with indefinite linear part,  J. Differential Equations  96(1992), 89-115.



\bibitem{AlamaLi2} S. Alama, Y. Y. Li, On "Multibump" bound states for certain semilinear elliptic equations, Indiana Univ. Math. J. 41(1992), 983-1026.



 \bibitem{BartschDingNon-Metrizable}   T. Bartsch, Y.H. Ding, Deformation theorems on non-metrizable vector spaces and applications to critical point theory, Math. Nachr. 279(2006), 1267-1288.



 \bibitem{Bartsch-Ding}   T. Bartsch, Y.H. Ding, On a nonlinear Schr$\ddot{\text{o}}$dinger equation with periodic potential, Math. Ann. 313
(1999), 15-37.





\bibitem{BC2} C. J. Batkam, F. Colin, Generalized fountain theorem and applications
to strongly indefinite semilinear problems, J. Math. Anal. Appl. 405(2013), 438-452.


\bibitem{Bernini-Bieganowski} F. Bernini, B. Bieganowski, Generalized linking-type theorem with applications to strongly indefinite problems
with sign-changing nonlinearities, Calc. Var. Partial Differential Equations 61(2022), 182.



  \bibitem{BieganowskiSurvey} B. Bieganowski, \text{S}chr$\ddot{\text{o}}$dinger-type equations with sign-changing nonlinearities: a survey,
preprint, arXiv:1810.01754.



  \bibitem{Bieganowski} B. Bieganowski, Solutions of the fractional \text{S}chr$\ddot{\text{o}}$dinger equation with a sign-changing
      nonlinearity, J. Math. Anal. Appl. 450(2017), 461-479.



  \bibitem{Bieganowski-Mederski} B. Bieganowski, J. Mederski, Nonlinear \text{S}chr$\ddot{\text{o}}$dinger equations with sum of periodic and vanishing
         potentials and sign-changing nonlinearities, Commun. Pure Appl. Anal. 17(2018), 143-161.





  \bibitem{Chabrowski} J. Chabrowski, A. Szulkin, On a semilinear Schr$\ddot{\text{o}}$dinger equation with critical Sobolev exponent,
Proc. Amer. Math. Soc. 130 (2002), 85-93.




  \bibitem{Chen-Tang} S. Chen, X. Tang, On the planar Schr$\ddot{\text{o}}$dinger equation with
     indefinite linear part and critical growth nonlinearity, Calc. Var. Partial Differential Equations 60(2021), 95.


\bibitem{CH} S. Chen, C. Wang, An infite-dimensional linking theorem without upper semi-continuous
assumption and its applications, J. Math. Anal. Appl. 420(2014), 1552-1567.



\bibitem{Chen-Chquard} S. Chen, L. Xiao, Existence of a nontrivial solution for a strongly indefinite periodic \text{C}hoquard system, Calc. Var. Partial Differential Equations 54(2015), 599-614.



\bibitem{cos} D. G. Costa, H. Tehrani, Existence and multiplicity results for a class of Schr$\ddot{\text{o}}$dinger equations with indefinite nonlinearities, Adv. Differential Equations 8(2003), 1319-1340.



\bibitem{V-Rabinowitz1} V. Coti Zelati, P. H. Rabinowitz, Homoclinic orbits for second order Hamiltonian systems possessing superquadratic potentials, J. Amer. Math. Soc.  4(1991), 693-727.



\bibitem{V-Rabinowitz2} V. Coti Zelati, P. H. Rabinowitz, Homoclinic type solutions for a semilinear elliptic PDE on $R^n$, Comm. Pure Appl. Math. 45(1992), 1217-1269.





\bibitem{Ding} Y. H. Ding,  Variational Methods   for Strongly Indefinite Problems, World Scientific, 2007.




 \bibitem{Ding-Lee}    Y. H. Ding, C. Lee,  Multiple solutions of Schr$\ddot{\text{o}}$dinger equations with indefinite linear part and super or asymptotically linear terms,  J. Differential Equations  222(2006), 137-163.



%%\bibitem{Ding-multiple} Y.H. Ding, S.X. Luan, Multiple solutions for a class of nonlinear Schr$\ddot{\text{o}}$dinger equations,
%%          J. Differential Equations 207 (2004), 423¨C457.




  \bibitem{OUR} L. J. Gu, H. S. Zhou, An improved fountain theorem and its application, Adv. Nonlinear Stud. 17(2017), 727-738.


  \bibitem{Hanqing}     Q. Han , F. Lin, Elliptic Partial Differential Equations,  Amer.  Math. Soc.,  New York, 2011.


\bibitem{SK} W. Kryszewski,  A. Szulkin, Generalized linking theorem with an application to semilinear Schr$\ddot{\text{o}}$dinger equation, Adv. Differential Equations 3(1998), 441-472.


\bibitem{Kuchment} P. Kuchment, The mathematics of photonic crystals, Mathem. Modeling in Opt. Sci.,
 207-272, SIAM, 2001.

\bibitem{Ligongbao}  G. Li ,  A. Szulkin, An asymptotically periodic \text{S}chr$\ddot{\text{o}}$dinger equation with indefinite linear part,
 Commun. Contemp. Math. 4 (2002), 763-776.





\bibitem{FLiu-JianfuYang}  F. Liu, J. Yang, Nontrivial solutions of \text{S}chr$\ddot{\text{o}}$dinger equations with indefinite nonlinearities,
 J. Math. Anal. Appl. 334 (2007), 627-645.





\bibitem{SI1} S. Liu, Z. Shen, Generalized saddle point theorem and asymptotically linear problems with periodic potential, Nonlinear Anal: Theory, Methods \& Applications, 86(2013), 52-57.



\bibitem{Mederski1}  J. Mederski, Ground states of a system of nonlinear \text{S}chr$\ddot{\text{o}}$dinger equations with periodic potentials,
    Comm. Partial Differential Equations 41(2016), 1426-1440.


\bibitem{Mederski2}  J. Mederski, Solutions to a nonlinear \text{S}chr$\ddot{\text{o}}$dinger equations with periodic potentials and zero on the boundary of
  the spectrum, Topol. Methods Nonlinear Anal. 46(2015), 755-771.


\bibitem{DLMills} D. L. Mills, Nonlinear Optics: Basic Concepts, Springer, 2012.


\bibitem{NieW} W. Nie, Optical nonlinearity: phenomena, applications, and materials, Adv. Mater. 5(1993), 520-545.





\bibitem{Pankov} A. Pankov, Periodic nonlinear \text{S}chr$\ddot{\text{o}}$dinger equation with application to photonic crystals, Milan J. Math. 73(2005),
  259-287.



\bibitem{TangxianhuaJDE} D. Qin,  V. R$\check{a}$dulescu,  X. Tang, Ground states and geometrically distinct solutions for periodic \text{C}hoquard-\text{P}ekar equations, J. Differential Equations 275(2021), 652-683.






   \bibitem{Minimax-methods} P. H. Rabinowitz,  Minimax Methods in Critical Point Theory with Applications to Differential Equations,
                 No. 65, Amer.  Math. Soc., 1986.


\bibitem{Reed-Simon}  M. Reed, B. Simon,  Methods of Modern Mathematical Physics, Vol. IV,
Academic Press, New York, 1975.


\bibitem{Rup} H. J. Ruppen, A generalized min-max theorem for functionals of strongly indefinite sign, Calc. Var. Partial Differential Equations 50(2014), 231-255.



\bibitem{weak-linking} M. Schechter, W. M. Zou, Weak linking theorems and Schr$\ddot{\text{o}}$dinger equations with critical Sobolev exponent, ESAIM  Control Optim. Calc. Var. 9(2003),  601-619.

\bibitem{SzulkinWeth} A. Szulkin, T. Weth,   Ground state solutions for some indefinite variational problems, J. Funct. Anal. 257(2009), 3802-3822.


\bibitem{MW} M. Willem, Minimax Theorems, Birkhauser, Boston, 1996.










 \end{thebibliography}
\end{document}